\documentclass[11pt]{article}

\usepackage{amsfonts,amsmath,enumerate}
\usepackage{amsthm}

\newtheorem{theorem}{Theorem}[section]
\newtheorem{lemma}{Lemma}[section]
\newtheorem{proposition}{Proposition}[section]

\newtheorem{corollary}{Corollary}[section]

\newtheorem{definition}{Definition}[section]

\title{\vspace{-0.7cm}Ramsey numbers of cubes versus cliques}
\date{}
\author{
David Conlon \thanks{Mathematical Institute, Oxford OX1 3LB,
United Kingdom. Email: david.conlon@maths.ox.ac.uk. Research
supported by a Royal Society University Research Fellowship.}
\and
Jacob Fox \thanks{Department of Mathematics, MIT, Cambridge,
MA 02139-4307. Email: fox@math.mit.edu. Research supported by
a Packard Fellowship, a Simons Fellowship, an MIT NEC Corp. award and NSF grant DMS-1069197.}
\and
Choongbum Lee \thanks{Department of Mathematics, MIT, Cambridge,
MA 02139-4307. Email: cb\_lee@mit.edu.
Research supported in part by a Samsung Scholarship.}
\and
Benny Sudakov \thanks{Department of Mathematics, ETH, 8092 Zurich, 
Switzerland and Department of Mathematics, UCLA, Los Angeles, CA 90095.
Email: benjamin.sudakov@math.ethz.ch. Research supported in part by SNSF 
grant 200021-149111 and by a USA-Israel BSF grant. } }

\begin{document}
\maketitle

\begin{abstract}
The cube graph $Q_n$ is the skeleton of the $n$-dimensional cube. It
is an $n$-regular graph on $2^n$ vertices.  The Ramsey number
$r(Q_n,K_s)$ is the minimum $N$ such that every graph of order $N$
contains the cube graph $Q_n$ or an independent set of order $s$.
In 1983, Burr and Erd\H{o}s asked whether the simple lower bound
$r(Q_n,K_s) \geq (s-1)(2^n - 1)+1$ is tight for $s$ fixed and $n$
sufficiently large. We make progress on this problem, obtaining the first
upper bound which is within a constant factor of the lower bound.
\end{abstract}

\section{Introduction}

For graphs $G$ and $H$, the {\it Ramsey number} $r(G,H)$ is defined to be the smallest
natural number $N$ such that every red/blue edge-coloring of the complete graph $K_N$ on $N$ vertices contains a red  copy of $G$ or a blue copy of $H$. 

One obvious construction, noted by Chv\'atal and Harary \cite{CH72}, which gives a lower bound for these numbers is to take $\chi(H) - 1$ disjoint red cliques of size $|G|-1$ and to connect every pair of vertices which are in different cliques by a blue edge. If $G$ is connected, the resulting graph contains neither a red copy of $G$ nor a blue copy of $H$, so that $r(G,H) \geq (|G|-1)(\chi(H) - 1) + 1$. Burr \cite{B81} strengthened this bound by noting that if $\sigma(H)$ is the smallest color class in any $\chi(H)$-coloring of the vertices of $H$, we may add a further red clique of size $\sigma(H) - 1$, obtaining 
\[r(G,H) \geq (|G|-1)(\chi(H) - 1) + \sigma(H).\]
Following Burr and Erd\H{o}s \cite{B81, BE83}, we say that a graph $G$ is {\it $H$-good} if the Ramsey number $r(G,H)$ is equal to this bound. If $\mathcal{G}$ is a family of graphs, we say that $\mathcal{G}$ is $H$-good if all sufficiently large graphs in $\mathcal{G}$ are $H$-good. When $H = K_s$, where $\sigma(K_s) = 1$, we simply say that $G$ or $\mathcal{G}$ is {\it $s$-good}.

The classical result on Ramsey goodness, which predates the definition, is the theorem of Chv\'atal \cite{C77} showing that all trees are $s$-good for any $s$. On the other hand, the family of trees is not $H$-good for every graph $H$. For example \cite{BEFRS89}, a construction of $K_{2,2}$-free graphs due to Brown \cite{B66} allows one to show that there is a constant $c < \frac{1}{2}$ such that 
\[r(K_{1,t}, K_{2,2}) \geq t + \sqrt{t} -  t^c\] 
for $t$ sufficiently large. This is clearly larger than $(|K_{1,t}|-1)(\chi(K_{2,2}) - 1) + \sigma(K_{2,2}) = t + 2$. 

In an effort to determine what properties contribute to being Ramsey good, Burr and Erd\H{o}s \cite{B87, BE83} conjectured that if $\Delta$ is fixed then the family of graphs with bounded maximum degree $\Delta$ is $s$-good for any $s$ (and perhaps even $H$-good for all $H$). This conjecture holds for bipartite graphs $H$ \cite{BEFRS85} but is false in general, as shown by Brandt \cite{B96}. He proved that for $\Delta \geq \Delta_0$ almost every $\Delta$-regular graph on a sufficiently large number of vertices is not even $3$-good. His result (and a similar result in \cite{NR09}) actually proves something stronger, namely, that if a graph $G$ has strong expansion properties then it cannot be $3$-good.

On the other hand, it has been shown that if a family of graphs exhibits poor expansion properties then it will tend to be good \cite{ABS12, NR09}. To state the relevant results, we define the {\it bandwidth} of a graph $G$ to be the smallest number $\ell$ for which there exists an ordering $v_1, \dots, v_n$ of the vertices of $G$ such that every edge $v_i v_j$ satisfies $|i - j| \leq \ell$. This parameter is known to be intimately linked to the expansion properties of the graph. In particular, any bounded-degree graph with poor expansion properties will have sublinear bandwidth \cite{BPTW10}.

The first such result, shown by Burr and Erd\H{o}s \cite{BE83}, states that for any fixed $\ell$ the family of connected graphs with bandwidth at most $\ell$ is $s$-good for any $s$. This result was recently extended by Allen, Brightwell and Skokan \cite{ABS12}, who showed that the set of connected graphs with bandwidth at most $\ell$ is $H$-good for every $H$. Their result even allows the bandwidth $\ell$ to grow at a reasonable rate with the size of the graph $G$. If $G$ is known to have bounded maximum degree, their results are particularly strong, saying that for any $\Delta$ and any fixed graph $H$ there exists a constant $c$ such that if $G$ is a graph on $n$ vertices with maximum degree $\Delta$ and bandwidth at most $c n$ then $G$ is $H$-good.

Many of the original problems of Burr and Erd\H{o}s \cite{BE83} have now been resolved \cite{NR09} but one that remains open is to determine whether the family of hypercubes is $s$-good for every $s$. The {\it hypercube} $Q_n$ is the graph on vertex set $\{0,1\}^n$ where two vertices are connected by an edge if and only if they differ in exactly one coordinate. This family of graphs has sublinear bandwidth but does not have bounded degree, so the result of Allen, Brightwell and Skokan does not apply. 

To get a first bound for $r(Q_n, K_3)$, note that a simple greedy embedding implies that any graph with maximum degree $d$ and at least $d n + 2^n$ vertices has a copy of $Q_n$ in its complement. Suppose now that the edges of a complete graph have been $2$-colored in red and blue and there is neither a blue triangle nor a red copy of $Q_n$. Then, since the blue neighborhood of any vertex forms a red clique, the maximum degree in blue is at most $2^n - 1$. Hence, the graph must have at most $(2^n - 1)n + 2^n < 2^n(n+1)$ vertices. We may therefore conclude that $r(Q_n, K_3) \leq 2^n (n+1)$. 

It is not hard to extend this argument to show that for any $s$ there exists a constant $c_s$ such that $r(Q_n, K_s) \leq c_s  2^n n^{s-2}$. This is essentially the best known bound. Here we improve this bound, obtaining the first upper bound which is within a constant factor of the lower bound.

\begin{theorem}
For any natural number $s \geq 3$, there exists a constant $c_s$ such that
\[r(Q_n, K_s) \leq c_s 2^n.\]
\end{theorem}

The original question of Burr and Erd\H{o}s \cite{BE83} relates to $s$-goodness but it is natural to also ask whether the family of cubes is $H$-good for any $H$. For bipartite $H$, this follows directly from a result of Burr, Erd\H{o}s, Faudree, Rousseau and Schelp \cite{BEFRS85}. Our result clearly implies that for any $H$, there is a constant $c_H$ such that $r(Q_n, H) \leq c_H 2^n$.

For triangles, the rough idea of the proof is to show that if a red/blue edge-coloring of $K_N$ does not contain a blue triangle then it may be tiled with a collection of red cliques which have low blue density between them. A red copy of $Q_n$ may then be found by inserting subcubes into each of the red cliques and patching them together using the fact that the blue density between these different cliques is low. Throughout the argument, it is very important to keep close control on the size and number of the red cliques as the definition of low blue density in a particular bipartite subgraph will depend on the size of the two cliques forming its endpoints.

For $K_4$, the method is an extension of this idea, except the tiling will now consist of cliques which are free of blue triangles, again with low blue density between different cliques. In order to make this useful for embedding cubes, we then have to perform a second level of tiling, splitting each such clique into red subcliques with low density between them, as was already done for triangles. In general, when we consider $K_s$, there will be $s - 2$ levels of tiling to keep track of. This makes the bookkeeping somewhat complex. Accordingly, we have chosen to present the proof in stages, first considering triangles, then $K_4$ and only then the general case. Although this leads to some redundancies, it allows us to introduce the additional concepts needed for each step at a reasonable pace. 

More generally, one may ask which families of graphs are $s$-good for all $s$. A powerful result proved by Nikiforov and Rousseau \cite{NR09} shows that graphs with small separators are $s$-good. They used this result to resolve a number of the original questions of Burr and Erd\H{o}s \cite{BE83} regarding Ramsey goodness. Let the {\it degeneracy} $d(G)$ of a graph $G$ be the smallest natural number $d$ such that every induced subgraph of $G$ has a vertex of degree at most $d$. Furthermore, we say that a graph $G$ has a $(t, \eta)$-separator if there exists a vertex subset $T \subseteq V(G)$ such that $|T| \leq t$ and every connected component of $V(G)\char92T$ has size at most $\eta |V(G)|$. The result of Nikiforov and Rousseau then says that for any $s \geq 3$, $d \geq 1$ and $0 < \gamma < 1$, there exists $\eta > 0$ such that the class $\mathcal{G}$ of $d$-degenerate graphs $G$ with a $(|V(G)|^{1-\gamma}, \eta)$-separator is $s$-good.

We will apply this theorem, together with the Alon-Seymour-Thomas separator theorem for graphs with a forbidden minor \cite{AST90}, to show that for any $s \geq 3$ any family of graphs with a forbidden minor is $s$-good. A graph $H$ is said to be a {\it minor} of $G$ if $H$ can be obtained from a subgraph of $G$ by contracting edges. By an {\it $H$-minor} of $G$, we mean a minor of $G$ which is isomorphic to $H$. For a graph $H$, let $\mathcal{G}_H$ be the family of graphs which do not contain an $H$-minor. 

\begin{theorem} \label{forbidminor}
For every fixed graph $H$, the class $\mathcal{G}_H$ of graphs $G$ which do not contain an $H$-minor is $s$-good for all $s \geq 3$.
\end{theorem}

In particular, since the family of planar graphs consists exactly of those graphs which do not contain $K_5$ or $K_{3,3}$ as a minor, we have the following corollary.

\begin{corollary}
The family of planar graphs is $s$-good for all $s \geq 3$.
\end{corollary}

 A {\it minor-closed} family $\mathcal{G}$ is a collection of graphs which is closed under taking minors. The graph minor theorem of Robertson and Seymour \cite{RS04} states that any minor-closed family of graphs may be characterized by a finite collection of forbidden minors. We say that a minor-closed family is nontrivial if it is not the class consisting of all graphs. Note that the following corollary is an immediate consequence of Theorem \ref{forbidminor} as any nontrivial minor-closed family $\mathcal{G}$ is a subfamily of $\mathcal{G}_H$, where $H$ is a graph not in $\mathcal{G}$. 

\begin{corollary} \label{minorclosed}
Any nontrivial minor-closed family of graphs $\mathcal{G}$ is $s$-good for all $s \geq 3$.
\end{corollary}

We will begin, in Section \ref{sec:K3}, by studying the Ramsey number of cubes versus triangles. We will then show, in Section \ref{sec:K4}, how our arguments extend to $K_4$ before treating the general case in Section \ref{sec:general}. In Section \ref{sec:separator}, we will prove Theorem \ref{forbidminor}. We will conclude with some further remarks.  All logarithms are base $2$ unless otherwise indicated. For the sake of clarity of presentation, we systematically omit floor and ceiling signs whenever they are not crucial. We also do not make any serious attempt to optimize absolute constants in our statements and proofs.

\section{Triangle versus cube} \label{sec:K3}

The argument works for $n \ge 6$. Consider a coloring of the edges of the complete graph $K_N$ on
the vertex set $[N] = \{1,2,\ldots,N\}$ for $N \ge 7000 \cdot 2^n$
with two colors, red and blue, and assume that there are no blue
triangles. We will prove that this coloring contains a red $Q_n$.

\subsection{Preprocessing the coloring} \label{sec:K3_preprocess}

For each $d=0, 1, 2, \ldots, \log n + 3$, we use the following
procedure to construct a family $\mathcal{S}$ of subsets of $[N]$
(note that $\log n + 3 \le n$ for $n \ge 6$):

\begin{quote}
If there exists a set $S$ which induces a red clique of order
exactly $4 \cdot 2^{n-d}$, then arbitrarily choose one, add it to
the family $\mathcal{S}$, and remove the vertices of the clique from
$[N]$. We define the \emph{codimension} $d(S)$ of such a set as $d(S)
= d$. When there are no more such red cliques, continue to the next
value of $d$. In the end, if we have $\sum_{S \in \mathcal{S}} |S| \ge \frac{N}{2}$, then
we let $\mathcal{S}$ be our family.
Otherwise, if $\sum_{S \in \mathcal{S}} |S| < \frac{N}{2}$, we add the set of
remaining vertices to $\mathcal{S}$, and declare it to have
codimension zero (note that this set has size at least $\frac{N}{2} \ge 4 \cdot 2^{n}$).
\end{quote}

In either of the cases, we have $\sum_{S \in \mathcal{S}} |S| \ge \frac{N}{2}$.
Note that we also have the following properties.

\begin{proposition} \label{prop:K3_preprocess}
\begin{enumerate}[(i)]
  \setlength{\itemsep}{1pt} \setlength{\parskip}{0pt}
  \setlength{\parsep}{0pt}
\item For an integer $i \ge 1$, let $X = \bigcup_{S \in \mathcal{S}, d(S) \ge i}S$. Then each vertex
$v \in [N]$ has at most $2^{n-i+3}$ blue neighbors in $X$.
\item For every set $S \in \mathcal{S}$, the subgraph induced by $S$ has maximum blue degree at most $\frac{2^{n - d(S)}}{2n}$.
\end{enumerate}
\end{proposition}
\begin{proof}
(i) Given $i \ge 1$ and a set $X$ defined as above, note that no subset of $X$ induces a
red clique of size at least $4 \cdot 2^{n-i+1} = 2^{n-i+3}$, since such a set would
have been added to $\mathcal{S}$ in the previous round. On the other hand,
since there are no blue triangles, the blue neighborhood
of every vertex induces a red clique. Therefore, for every $v \in [N]$, $v$ has at most
$2^{n-i+3}$ blue neighbors in $X$.

(ii) The claim is trivially true for $d >0$, since every set of
codimension at least 1 is a red clique. For $d=0$, if $S$ is not a
red clique, 
then it is the set of codimension zero added to $\mathcal{S}$
after the sets of codimensions $0$ to $\log n + 3$. 
Thus, there are no subsets of $S$ which induce a red
clique of size at least $4 \cdot 2^{n-(\log n+3)}$, since
such a set would have been added earlier to the family $\mathcal{S}$.
On the other hand, since there are no blue triangles, the blue
neighborhood of every vertex induces a red clique. Therefore, every
vertex of $S$ has at most $4 \cdot 2^{n- (\log n+3)} \le
\frac{2^n}{2n}$ blue neighbors in $S$.
\end{proof}

\subsection{Tiling the cube} \label{sec:K3_tiling}

Our strategy is to first decompose the cube $Q_n$ into smaller cubes
and then to embed it piece by piece into $K_N$, placing one of the subcubes in each of the subsets from $\mathcal{S}$. We represent the subcubes of $Q_n$
by using vectors in $\{0,1,*\}^{n}$. For example, the vector $(0, 1, *, \ldots, *)$
will represent the subcube $\{ (0, 1, x_1, \ldots, x_{n-2}) : (x_1, \ldots, x_{n-2}) \in \{0,1\}^{n-2} \}$.

We seek a \emph{tiling} of $Q_n$,
which we define as a collection of vertex-disjoint cubes that covers
all the vertices of $Q_n$ and uses only cubes that have all their fixed coordinates
at the start of the vector. That is, they are of the form $C = (a_1, a_2, \ldots, a_d, *, *, \ldots, *)$
for some $d \ge 0$ and $a_1, a_2,\ldots, a_d \in \{0,1\}$. We will call such a cube {\it special} and let
$d$ be the \emph{codimension} $d(C)$ of $C$. We say that two disjoint cubes $C$ and $C'$
are \emph{adjacent}, if there are vertices $x \in C$ and $x' \in C'$ which are adjacent in $Q_n$.
Note that the following list of properties holds for such cubes and a tiling composed from them.

\begin{proposition} \label{prop:K3_cuberelation}
\begin{enumerate}[(i)]
  \setlength{\itemsep}{1pt} \setlength{\parskip}{0pt}
  \setlength{\parsep}{0pt}
\item If two special cubes $C$ and $C'$ intersect, then we have $C \subset C'$ or $C' \subset C$.
\item Two disjoint cubes $C = (a_1, \ldots, a_d, *, \ldots, *)$ and $C' = (b_1, \ldots, b_{d'}, *, \ldots, *)$ for $d' \le d$ are adjacent if and only if $(a_1, \ldots, a_{d'})$ and $(b_1, \ldots, b_{d'})$ differ in
exactly one coordinate.
\item For a tiling $\mathcal{C}$ and a special cube $C \in \mathcal{C}$ of codimension $d$, there are at most $d$ other special cubes $C' \in \mathcal{C}$ of codimension $d(C') \le d$ which are adjacent to $C$.
\end{enumerate}
\end{proposition}
\begin{proof}
(i) Let $d=d(C)$, $d'=d(C')$ and, without loss of generality, suppose that $d' \le d$. Suppose that $(a_1, \ldots, a_n)\in C \cap C'$. Then, since we are only considering special cubes, we must
have $C = (a_1, \ldots, a_d, *, \ldots, *)$ and $C' = (a_1, \ldots, a_{d'}, *,\ldots,*)$.
However, we then have $C \subset C'$.

(ii) If $(a_1, \ldots, a_{d'}) = (b_1, \ldots, b_{d'})$, then we have $C \subset C'$, and thus we may assume that
this is not the case. Cubes $C$ and $C'$ are adjacent if and only if there are two vectors
$\mathbf{v} = (a_1, \ldots, a_d, x_{d+1}, \ldots, x_{n})$ and $\mathbf{w} = (b_1, \ldots, b_{d'}, y_{d'+1}, \ldots, y_{n})$ in
$\{0,1\}^n$ which differ in exactly one coordinate. However, since the two vectors restricted to the
first $d'$ coordinates are already different, this can happen only if $(a_1, \ldots, a_{d'})$
and $(b_1, \ldots, b_{d'})$ differ in exactly one coordinate. Moreover, one can see that if this is
the case, then there is an assignment of values to $x_i$ and $y_j$ so that $\mathbf{v}$ and
$\mathbf{w}$ indeed differ in exactly one coordinate.

(iii) Since $C$ has codimension $d$, there are only $d$ coordinates that one can `flip' from $C$ to obtain a cube adjacent to $C$.
\end{proof}

As we have already mentioned, our tiling $\mathcal{C}$ of the cube $Q_n$ will be constructed in correspondence with the family $\mathcal{S}$ constructed in Section \ref{sec:K3_preprocess}. 
We will construct the tiling $\mathcal{C}$
by finding cubes of the tiling one by one. We slightly abuse
notation and use $\mathcal{C}$ also to denote the `partial' tiling,
where only part of the cube $Q_n$ is covered. At each step, we will find
a subcube $C$ which covers some non-covered part of $Q_n$ and
assign it to some set $S_C \in \mathcal{S}$. We say that such an
assignment is \emph{proper} if the following properties hold.

\begin{quote}
{\bf Proper assignment.}

1. If $C$ has codimension $d$, then $S_C$ has codimension $d$.

2. Suppose that $C$ is adjacent to some cube $C'$ already in the
tiling and that $C'$ is assigned to $S_{C'}$. Then the bipartite
graph induced by $S_C$ and $S_{C'}$ contains at most
$\frac{|S_C||S_{C'}|}{16\delta^2}$ blue edges, where $\delta=
\max\{d(C), d(C')\}$.
\end{quote}

Our algorithm for finding the tiling $\mathcal{C}$ and the
corresponding sets in $\mathcal{S}$ is as follows.

\begin{quote}
\textbf{Tiling Algorithm.} At each step, consider all possible special cubes
$C$ which 
\begin{itemize}
\item[(a)] are disjoint from the cubes $C' \in \mathcal{C}$, 

\item[(b)]
have $d(C) \ge d(C')$ for all $C' \in \mathcal{C}$, and 

\item[(c)] for which
there exists a set $S \in \mathcal{S}$ which has not yet been
assigned and such that assigning $C$ to $S$ is a proper assignment. 
\end{itemize}

Take a cube $C_0$ of minimum codimension satisfying (a), (b) and (c) and add it
to the tiling. Assign $C_0$ to the set $S_{C_0} \in \mathcal{S}$
given by (c).
\end{quote}

The following proposition shows that the algorithm will terminate successfully.
\begin{proposition} \label{prop:K3_tiling}
If the tiling is not complete, then the algorithm always chooses a cube from
a non-empty collection.
\end{proposition}
\begin{proof}
If $\mathcal{S}$ contains a set of codimension
zero, then the algorithm will choose the whole cube $Q_n$ in the
first step and assign it to a set of codimension zero.
In this case we have a trivial tiling and there is nothing to prove.

Thus, we may assume that $\mathcal{S}$ contains no set of codimension zero.
If this is the case, note that for all $C \in \mathcal{C}$, we have $|S_C| = 4|C|$. Let
$\mathcal{S}''$ be the subfamily of $\mathcal{S}$ consisting of sets
which are already assigned to some cube in $\mathcal{C}$ and let
$\mathcal{S}' = \mathcal{S} \setminus \mathcal{S}''$. Note that
\begin{align*}
   \Big| \bigcup_{S\in \mathcal{S}'} S \Big|
   = \Big|\bigcup_{S\in \mathcal{S}} S\Big| - \Big|\bigcup_{S \in \mathcal{S}''} S \Big|
   \ge \frac{N}{2} - \Big|\bigcup_{C \in \mathcal{C}} S_C \Big| \ge 3500 \cdot 2^{n} - 4 \cdot 2^{n} \ge 3496 \cdot 2^n.
\end{align*}
For each $i$, let $\mathcal{S}'_i = \{S \in \mathcal{S}' : d(S) = i\}$. If 
$|\mathcal{S}'_i| \le 32i^3$ for all $i$, then we have
\begin{align*}
\left| \bigcup_{S \in \mathcal{S}'} S \right|
   = \sum_{i} \left| \bigcup_{S \in \mathcal{S}'_i} S \right|
   &\le \sum_{i} 32i^3 \cdot 4 \cdot 2^{n-i}
   = 128 \cdot 2^{n} \sum_{i} \frac{i^3}{2^{i}}
   = 3328 \cdot 2^{n},
\end{align*}
which is a contradiction (we used the fact that $\sum_{i} \frac{i^3}{2^i} = 26$). 
Therefore, there exists an index $i$ for which $|\mathcal{S}'_i| > 32i^3$.
We have $i > 0$ since $\mathcal{S}$ contains no set of codimension zero.

Since the tiling is not complete, there exists a
vertex $(a_1, \ldots, a_n) \in Q_n$ which is not yet covered by any
of the cubes already in $\mathcal{C}$. Consider the cube $C = (a_1,
\ldots, a_i, *, \ldots, *)$. By Proposition
\ref{prop:K3_cuberelation}, there are at most $i$ cubes in our
partial tiling $\mathcal{C}$ of codimension at most $i$ which are
adjacent to this cube. Let $C_1, \ldots, C_j$ for $j\le i$ be these
cubes and fix an index $a \le j$. We say that a set $S \in
\mathcal{S}_i'$ is \emph{bad} for $S_{C_a}$ if there are at least
$\frac{1}{16i^2}|S_{C_a}||S|$ blue edges between $S_{C_a}$ and $S$.
We claim that there are at most $32i^2$ sets in $\mathcal{S}'_i$
which are bad for $S_{C_a}$. 

Let $X_i = \bigcup_{S \in
\mathcal{S}'_i} S$ and note, by Proposition
\ref{prop:K3_preprocess}(i), that there are at most $|S_{C_a}|\cdot
2^{n-i+3}$ blue edges between the sets $S_{C_a}$ and $X_i$. Each set
$S \in \mathcal{S}'_i$ which is bad for $S_{C_a}$ accounts for at
least $\frac{1}{16i^2}|S_{C_a}||S| =
\frac{2^{n-i+3}}{32i^2}|S_{C_a}|$ such blue edges (note that $|S| = 4\cdot 2^{n-i}$). Therefore, in
total, there are at most $32i^2$ sets in $\mathcal{S}'_i$ which are
in bad relation with $S_{C_a}$, as claimed above. Since
$|\mathcal{S}'_i| > 32i^3$, there exists a set $S \in
\mathcal{S}'_i$ which is not bad for any of the sets
$S_{C_a}$ for $1 \leq a \le j \le i$.

Suppose that in the previous step we embedded some cube of codimension $d$.
In order to show that $C$ satisfies (a) and (b), it suffices to
verify that $i \ge d$, since $i \ge d$ implies the fact that $C$ is
disjoint from all the other cubes in $\mathcal{C}$ (if $C$
intersects some other cube, then that cube must contain $C$ and
therefore also contains $(a_1, \ldots, a_n)$ by Proposition
\ref{prop:K3_cuberelation}). Furthermore, if this is the case,
assigning $C$ to $S$ is a proper assignment (thus we have
(c)), since $C_1, \ldots, C_j$ are the only cubes adjacent to $C$ already in the tiling.

Suppose, for the sake of contradiction, that $i < d$
and consider the time $t$ immediately after we last embedded a cube
of codimension at most $i$. At time $t$, since $(a_1, \ldots, a_n)$ was not
covered, the cubes in $\mathcal{C}$ are disjoint from $C$. Moreover,
the set of cubes of codimension at most $i$ which are adjacent to $C$ is the
same as at the current time. Therefore, $C$ could have been added to the
tiling at time $t$ as well, contradicting the fact that we
always choose a cube of minimum codimension. Thus we have $i \ge d$,
as claimed.
\end{proof}

Note that as an outcome of our algorithm, we obtain a tiling $\mathcal{C}$ for which every pair of adjacent cubes $C, C' \in \mathcal{C}$ have  at
most $\frac{1}{16 \delta^2} |S_C||S_{C'}|$ blue edges between $S_C$ and $S_{C'}$, where 
$\delta = \max\{d(C), d(C')\}$.

\subsection{Imposing a maximum degree condition} \label{sec:K3_maxdegree}

Having constructed the tiling $\mathcal{C}$ and made the
assignment of the cubes to sets in $\mathcal{S}$, we
now wish to impose certain maximum degree conditions between the sets $S_C$ for
$C \in \mathcal{C}$. For a set $C \in \mathcal{C}$ of codimension $d
= d(C)$, let $C_1, \ldots, C_{j}$ be the cubes of codimension at
most $d$ which are adjacent to $C$. Note that we have $j \le d$ by
Proposition \ref{prop:K3_cuberelation}. Moreover, there are
at most $\frac{1}{16d^2}|S_C||S_{C_a}|$ blue edges between $S_C$ and
$S_{C_a}$ for all $a \le j$. 

Now, for each $a$, remove all the vertices in $S_C$ which have at
least $\frac{1}{8d}|S_{C_a}|$ blue neighbors in $S_{C_a}$ and let
$T_C$ be the subset of $S_C$ left after these removals. Since there
are at most $\frac{1}{16d^2}|S_C||S_{C_a}|$ blue edges between $S_C$
and $S_{C_a}$, for each index $a$, we remove at most
$\frac{|S_C|}{2d}$ vertices from $S_C$, and thus the resulting set
$T_C$ is of size at least $\frac{|S_C|}{2} \ge 2 \cdot 2^{n-d}$.  Note that all the vertices
in $T_C$ have blue degree at most $\frac{1}{8d}|S_{C_a}| \le
\frac{1}{4d}|T_{C_a}|$ in the set $T_{C_a}$. That is, we have the
following property.

\begin{quote}
\textbf{Maximum degree condition.} For each pair of adjacent cubes $C, C' \in
\mathcal{C}$ with $d(C) \ge d(C')$, every vertex in $T_C$ has at
most $\frac{1}{4d(C)}|T_{C'}|$ blue neighbors in the set $T_{C'}$.
\end{quote}

\subsection{Embedding the cube}

We now show how to embed $Q_n$. Recall that $Q_n$ was tiled by cubes
in $\mathcal{C}$, each corresponding to a subset from the family $\mathcal{S}$. We will greedily embed these cubes one by one into their assigned sets from the family $\mathcal{S}$, in
decreasing order of their codimensions. If there are several cubes
of the same codimension, then we arbitrary choose the order between
them. 

For each $C \in \mathcal{C}$, we will greedily embed the vertices of $C$ into the set $T_C \subseteq S_C$. Let $d = d(C)$. Suppose that we are about to embed $x \in C$ and let $f :
Q_n \rightarrow [N]$ denote the partial embedding of the cube $Q_n$
obtained so far. Note that $x$ has at most $d$ neighbors
$x_1, x_2, \ldots, x_j$ (for $j \le d$) which are already embedded and belong to a
cube other than $C$. Since we have only embedded cubes of codimension at
least $d$, the maximum degree
condition imposed in Section \ref{sec:K3_maxdegree} implies that all the vertices $f(x_i)$ have blue degree at most
$\frac{1}{4d}|T_C|$ in the set $T_C$. Together, these
neighbors forbid at most $\frac{1}{4}|T_C|$ vertices of
$T_C$ from being the image of $x$. 

In addition, $x$ has at most $n-d$
neighbors $y_1, \ldots, y_{k}$ (for $k \le n-d$) which are already embedded and
belong to $C$. By Proposition \ref{prop:K3_preprocess}, each vertex
$f(y_i)$ has blue degree at most $\frac{2^{n-d}}{2n}$ in the set
$T_C$. Together, these neighbors forbid at most
$\frac{2^{n-d}}{2}$ vertices of $T_C$ from being the image of $x$.
Finally, there are at most $2^{n-d}-1$ vertices in $T_C$ which are
images of some other vertex of $C$ that is already embedded. Therefore, the
number of vertices in $T_C$ into which we cannot embed $x$ is at most
\[ \frac{1}{4}|T_C| + \frac{1}{2} \cdot 2^{n-d} + (2^{n-d}-1), \]
which is less than $|T_C|$ since $|T_C| \ge 2 \cdot 2^{n-d}$. Hence,
there exists a vertex in $T_C$ which we can choose as an image of
$x$ to extend the current partial embedding of the cube. Repeating
this procedure until we embed the whole cube $Q_n$
completes the proof.

\section{Cliques of order 4} \label{sec:K4}

The argument for general cliques is similar to that for triangles
given in the previous section. However, there are several new concepts
involved. To slowly develop the necessary concepts, we first provide
a proof of the next case, which is $K_4$ versus a cube.

Recall that in the triangle case, we started by iteratively
finding sets which formed a red clique (see Section \ref{sec:K3_preprocess}).
Red cliques were a natural choice, since the blue
neighborhood of every vertex formed a red clique. Either we
were able to find large red cliques or we were able to restrict the maximum
blue degree of vertices in some way (see Proposition \ref{prop:K3_preprocess}).
If one attempts to employ the same strategy for the $K_4$ case, then
the natural choice of sets that we should take instead of red cliques
are blue triangle-free sets, since the blue neighborhood of every vertex
now forms a blue triangle-free set. Suppose that we have found a family $\mathcal{S}_{1}$
of blue triangle-free sets. Since blue $K_3$-free sets are not as powerful
as red cliques for embedding subgraphs, we repeat the whole argument
within each blue triangle-free set $S \in \mathcal{S}_{1}$,
to obtain red cliques which are subsets of $S$. By so doing, we obtain a second family of sets
$\mathcal{S}_{2}$, consisting of red cliques. We refer to sets in $\mathcal{S}_{\ell}$ as
\emph{level $\ell$} sets and, for a set $S \in \mathcal{S}_{\ell}$,
we define its level as $\ell(S) = \ell$.

In order to find an embedding of the cube $Q_n$ using a strategy similar to
that used in the triangle case,  we wish to find a tiling
of $Q_n$ and, for each cube $C$ in the tiling, a red clique $S_C \in \mathcal{S}_{2}$
so that for two adjacent cubes $C$ and $C'$, the sets $S_C$ and $S_{C'}$ stand in
`good' relation. However, directly finding such a tiling and an assignment
is somewhat difficult since we do not have good control on the blue edges
between the sets in $\mathcal{S}_{2}$. To be more precise, suppose that
we are given $S_1, S_1' \in \mathcal{S}_{1}$ and subsets
$S_2 \subset S_1, S_2' \subset S_1'$ in $\mathcal{S}_{2}$.
Then the control on the blue edges between $S_2$ and $S_2'$ is `inherited'
from the control on the blue edges between $S_1$ and $S_1'$, and thus depends
on the relative sizes of $S_1$ and $S_1'$, not on the relative sizes of $S_2$
and $S_2'$ as in the triangle case (unless $S_1 = S_1'$). To circumvent this difficulty, 
we will need to maintain tight control on the edge density between different sets in 
$\mathcal{S}_1$ as well as those in $\mathcal{S}_2$.

We seek a \emph{double tiling},
which is defined to be a pair
$\mathcal{C} = \mathcal{C}_{1} \cup \mathcal{C}_{2}$ of tilings satisfying
the property that for every $C_2 \in \mathcal{C}_{2}$ there exists a cube
$C_1 \in \mathcal{C}_{1}$ such that $C_2 \subset C_1$ (in other words,
$\mathcal{C}_{2}$ is a \emph{refined} tiling of $\mathcal{C}_{1}$).
We refer to cubes in $\mathcal{C}_{\ell}$ as \emph{level $\ell$} cubes
and, for a cube $C \in \mathcal{C}_{\ell}$, we define its level as $\ell(C) = \ell$.
Our goal is to find, for each $\ell=1,2$, an assignment of
cubes in $\mathcal{C}_{\ell}$ to sets in $\mathcal{S}_{\ell}$. The following are
the key new concepts involved in the $K_4$ case.

\begin{definition}
\begin{itemize}
\item[(i)] For a cube $C \in \mathcal{C}_{1}$, we define its \emph{1-codimension} as
$d_1(C) = d(C)$. For a cube $C \in \mathcal{C}_{2}$ contained in
a cube $C_1 \in \mathcal{C}_{1}$, we define its \emph{1-codimension} as
$d_1(C) = d(C_1)$ and its \emph{2-codimension} as $d_2(C) = d(C) - d(C_1)$.
\item[(ii)] Suppose $C, C' \in \mathcal{C}$ are adjacent. We say that they
have \emph{level 1 adjacency} if the level 1 cubes containing $C$ and $C'$
are different. Otherwise, they have \emph{level 2 adjacency} (note that
$C$ and $C'$ are both level 2 cubes in the second case).
\end{itemize}
\end{definition}

After finding the double tiling, in order to find an embedding of the cube,
we only consider the level 2 tiling (as explained
above, we need to go through a level 1 tiling in order to obtain a `nice' level 2 tiling).
For two adjacent cubes $C, C' \in \mathcal{C}_{2}$, the definition of the
corresponding sets $S_C, S_{C'}$ being in `good relation' (see the definition
of proper assignment in Section \ref{sec:K4_tiling}) now depends on
the type of adjacency they have. If they have level 1 adjacency, then
the good relation will be defined in terms of their 1-codimensions
and if they have level 2 adjacency, then it will be in terms of their 2-codimensions.
By doing this, we can overcome the above mentioned difficulty of not having
enough control on the blue edges between $S_C$ and $S_{C'}$. Afterwards, we
proceed as in the previous section, by imposing a maximum degree condition between
cubes in $\mathcal{C}_2$ and then embedding the cube.

We now provide the details of the proof.
Our argument works for $n \ge 32$. Consider a coloring of the edges of the complete graph $K_N$ on
the vertex set $[N]$ for $N \ge 2^{46} \cdot 2^n$
with two colors, red and blue, and assume that there is no blue
$K_4$. We prove that this coloring contains a red $Q_n$.

\subsection{Preprocessing the coloring} \label{sec:K4_preprocess}

Note that we have $n \ge 2\log n + 21$ for $n \ge 32$.
We use the following procedure to construct our first level $\mathcal{S}_{1}$ of subsets of $[N]$:

\begin{quote}
For each $d=0, 1, 2, \ldots, \log n + 18$,
if there exists a set $S$ which induces a blue triangle-free graph of order
exactly $2^{18} \cdot 2^{n-d}$, then arbitrarily choose one, add it to
the family $\mathcal{S}_{1}$, and remove the vertices of $S$ from
$[N]$. We define the \emph{1-codimension} $d_1(S)$ of such a set as $d_1(S)
= d$. When there are no more such sets, continue to the next
value of $d$. If, after running through all values of $d$,
$\sum_{S \in \mathcal{S}} |S| < \frac{N}{2}$, then add the set of
remaining vertices to $\mathcal{S}$ and declare it to be an \emph{exceptional set}
with 1-codimension zero (note that this set has size at least $\frac{N}{2} \ge 2^{18} \cdot 2^{n}$).
\end{quote}

Now we perform a similar decomposition for each set in $\mathcal{S}_{1}$ to construct
our second level $\mathcal{S}_{2}$.
For each set $S_1 \in \mathcal{S}_{1}$, consider the following procedure
(suppose that $S_1$ has codimension $d_1 = d(S_1)$):

\begin{quote}
1. If $S_1$ is not exceptional: for each $d=0, 1, 2, \ldots, \log n + 3$,
if there exists a subset $S \subset S_1$ which induces a red clique
of order exactly $8 \cdot 2^{n-d_1 - d}$, then arbitrarily choose one, add it to
the family $\mathcal{S}_{2}$, and remove the vertices of the clique from
$S_1$. We define the \emph{2-codimension} of such a set as $d_2(S)
= d$ and the \emph{1-codimension} as
$d_1(S) = d_1$. When there are no more such red cliques, continue to the next
value of $d$. If, after running through all values of $d$,
$\sum_{S \in \mathcal{S}_{2}, S \subset S_1} |S| < \frac{|S_1|}{2}$, then add the set of
remaining vertices to $\mathcal{S}_{2}$, and declare it to be an
\emph{exceptional set} with 2-codimension zero and 1-codimension $d_1$
(note that this set has size at least $\frac{|S_1|}{2} \ge 8 \cdot 2^{n-d_1}$).

2. If $S_1$ is exceptional: add $S_1$ to the family $\mathcal{S}_{2}$.
Define the $2$-codimension of $S_1 \in \mathcal{S}_{2}$ as zero and its $1$-codimension
as zero.
\end{quote}

Let $\mathcal{S} = \mathcal{S}_{1} \cup \mathcal{S}_{2}$ (we suppose that $\mathcal{S}$ is a multi-family and,
thus, if a set is in both $\mathcal{S}_{1}$ and $\mathcal{S}_{2}$, then we have two copies
of this set in $\mathcal{S}$, however they can be distinguished by their levels).
For a set $S \in \mathcal{S}$, define its \emph{codimension} as
$d(S) = \sum_{i \le \ell(S)} d_i(S)$.

\begin{proposition} \label{prop:K4_preprocess}
\begin{enumerate}[(i)]
  \setlength{\itemsep}{1pt} \setlength{\parskip}{0pt}
  \setlength{\parsep}{0pt}
\item $\sum_{S \in \mathcal{S}_{1}} |S| \ge \frac{N}{2}$ and, for every $S_1 \in \mathcal{S}_{1}$,
we have $\sum_{S \in \mathcal{S}_{2}, S \subset S_1} |S| \ge \frac{|S_1|}{2}$.
\item For an integer $i \ge 1$, let $X = \bigcup_{S \in \mathcal{S}_{1}, d_1(S) \ge i}S$. Then each vertex
$v \in [N]$ has at most $2^{19} \cdot 2^{n-i}$ blue neighbors in $X$.
\item For a set $S_1 \in \mathcal{S}_{1}$ and an integer $i \ge 1$, let
\[ X = \bigcup_{S \in \mathcal{S}_{2}, S \subset S_1, d_2(S) \ge i}S. \] Then each vertex $v \in S_1$
has at most $16 \cdot 2^{n-d_1(S_1)-i}$ blue neighbors in $X$.
\item For every set $S \in \mathcal{S}_{2}$, the subgraph induced by $S$ has maximum blue degree at most $\frac{2^{n-d(S)}}{n}$.
\end{enumerate}
\end{proposition}
\begin{proof}
Part (i) is clear and we omit the proof of (ii), since its proof is similar to that of part (i) of Proposition \ref{prop:K3_preprocess}.

(iii) Suppose that $S_1$ is an exceptional set. Then $S_1$ is the unique set $S$ 
satisfying $S \in \mathcal{S}_{2}$ and $S \subset S_1$. Since $d_2(S_1) = 0$,
$X$ is an empty set for every $i \ge 1$. Thus the conclusion follows.

Now suppose that $S_1$ is not an exceptional set. Since $S_1$ induces a
blue triangle-free set, a blue neighborhood in $S_1$ of a vertex $v \in S_1$
forms a red clique. Thus, if a vertex $v \in S_1$ has more than $16 \cdot 2^{n-d_1(S_1)-i}$
blue neighbors in $X$, then inside the neighborhood of $v$ we can find a
set inducing a red clique of size at least $8 \cdot 2^{n-d_1(S_1)-(i-1)}$. However, this
set must have been added in the previous step. Therefore, there are
no such vertices.

(iv) Suppose that $S \subset S_1$ with $S_1 \in \mathcal{S}_1$. If both $S$ and $S_1$
are not exceptional sets, then $S$ induces a red clique and the conclusion follows.
Suppose now that $S$ is exceptional and $S_1$ is not. Then, as in (iii), we can see
that there is no vertex in $S$ which has at least $8 \cdot 2^{n - d_1(S) - (\log n + 3)}
= \frac{2^{n-d(S)}}{n}$ blue neighbors in $S$.
One can similarly handle the case when $S_1$ is exceptional, since we have $S=S_1$ in this case.
\end{proof}

\subsection{Tiling the cube} \label{sec:K4_tiling}

The following proposition is similar to Proposition \ref{prop:K3_cuberelation} 
(we omit its proof).

\begin{proposition} \label{prop:K4_cuberelation}
Let $\mathcal{C} = \mathcal{C}_{1} \cup \mathcal{C}_{2}$ be a double tiling, and
let $C \in \mathcal{C}$ be a special cube of codimension $d$.
\begin{enumerate}[(i)]
  \setlength{\itemsep}{1pt} \setlength{\parskip}{0pt}
  \setlength{\parsep}{0pt}
\item If $C$ is a level 1 cube, then, for each $\ell=1, 2$, there are at most
$d_1(C)$ special cubes of level $\ell$ and codimension at most $d$ which are adjacent to $C$ (and $C$
has level $1$ adjacency with all these cubes).
\item If $C$ is a level 2 cube, then there are at most $d_2(C)$ special cubes of codimension at most $d$
which have level $2$ adjacency with $C$ (they are necessarily of level 2)
and, for $\ell=1,2$, at most $d_1(C)$ cubes of level $\ell$
and codimension at most $d$ which have level $1$ adjacency with $C$.
\end{enumerate}
\end{proposition}

Our double tiling $\mathcal{C}$ of the cube $Q_n$ will be constructed in
correspondence with the family $\mathcal{S}$ constructed in Section
\ref{sec:K4_preprocess}. We will construct $\mathcal{C}$
by finding cubes of the tiling one by one. We slightly abuse
notation and use $\mathcal{C}$ also to denote the `partial' double tiling,
where only part of the cube $Q_n$ is covered.
Ideally, we would like to construct $\mathcal{C}$ by
constructing the level 1 tiling $\mathcal{C}_{1}$ first and
then the level 2 tiling $\mathcal{C}_{2}$. However, as we will
soon see, it turns out that constructing the tiling in increasing
order of codimension is more effective than in increasing order of
level. At each step, we find
a subcube $C$ which covers some non-covered part of $Q_n$ and
assign it to some set $S_C \in \mathcal{S}$. We say that such an
assignment is \emph{proper} if the following properties hold.

\begin{quote}
{\bf Proper assignment.}

1. $\ell(C) = \ell(S_C)$ and, for $\ell = \ell(C)$, we have $d_\ell(C) = d_\ell(S_C)$.

2. If $C_2 \subset C_1$ with $C_1 \in \mathcal{C}_1$ and $C_2 \in \mathcal{C}_2$, then
$S_{C_2} \subset S_{C_1}$.

3. Suppose that $C$ is adjacent to some cube $C'$ which is already in the
tiling, where $C'$ is assigned to $S_{C'}$ and
$C$ and $C'$ have level $\rho$ adjacency. Then the number
of blue edges in the bipartite graph induced by $S_C$ and $S_{C'}$
is at most $|S_C||S_{C'}| \cdot (8\delta)^{\ell(C) + \ell(C') - 6}$, where
$\delta = \max\{d_\rho(C), d_\rho(C')\}$.
\end{quote}

Our algorithm for finding the tiling $\mathcal{C}$ and the
corresponding sets in $\mathcal{S}$ is as follows.

\begin{quote}
\textbf{Tiling Algorithm.} At each step, consider all possible special cubes
$C$ which 
\begin{itemize}
\item[(a)] can be added to $\mathcal{C}$ to extend the partial tiling, 

\item[(b)]
have $d(C) \ge d(C')$ for all $C' \in \mathcal{C}$, and 

\item[(c)] for which
there exists a set $S \in \mathcal{S}$ which has not yet been
assigned and such that assigning $C$ to $S$ gives a proper assignment. 
\end{itemize}
Take a
cube $C_0$ of minimum codimension satisfying (a), (b) and (c) and add it
to the tiling. Assign $C_0$ to the set $S_{C_0} \in \mathcal{S}$
given by (c).
\end{quote}

Condition (a) is equivalent to saying that either $C$ is disjoint from all
cubes in $\mathcal{C}_{1}$ or is contained in some cube in $\mathcal{C}_{1}$
and is disjoint from all cubes in $\mathcal{C}_{2}$.
The following proposition shows that the algorithm will terminate successfully.
\begin{proposition} \label{prop:K4_tiling}
If the tiling is not complete, then the algorithm always chooses a cube from
a non-empty collection.
\end{proposition}

\begin{proof}
Suppose that in the previous step we embedded some cube of codimension $d$
(let $d=0$ for the first iteration of the algorithm).
Since the tiling is not complete, there exists a
vertex $(a_1, \ldots, a_n) \in Q_n$ which is not covered twice.

\medskip

\noindent \textbf{Case 1}: $(a_1, \ldots, a_n)$ is not covered by any of
the cubes in $\mathcal{C}_{1}$.

Let $\mathcal{S}''$ be the subfamily of $\mathcal{S}_{1}$ consisting of sets
which are already assigned to some cube in $\mathcal{C}_{1}$, and let
$\mathcal{S}' = \mathcal{S}_{1} \setminus \mathcal{S}''$. If $\mathcal{S}''$
contains a set of codimension zero, then the corresponding cube is the whole
cube $Q_n$, contradicting the fact that
$(a_1, \ldots, a_n)$ is not covered. Thus we may assume that there is
no set of codimension zero in $\mathcal{S}''$, from which it follows that
$|S_C| = 2^{18} |C|$ for all $C \in \mathcal{C}_{1}$. Note that
\begin{align*}
   \Big| \bigcup_{S\in \mathcal{S}'} S \Big|
   = \Big|\bigcup_{S\in \mathcal{S}_{1}} S\Big| - \Big|\bigcup_{S \in \mathcal{S}''} S \Big|
   \ge \frac{N}{2} - \Big|\bigcup_{C \in \mathcal{C}} S_C \Big| \ge 2^{45} \cdot 2^{n} - 2^{18} \cdot 2^{n} > 2^{44} \cdot 2^n.
\end{align*}
For each $i$, let $\mathcal{S}'_i = \{S \in \mathcal{S}' : d(S) = i\}$. Suppose that
$|\mathcal{S}'_i| \le 2^{14} i^5$ for all $i$. Then, since each set in $\mathcal{S}_i'$
has size $2^{18}\cdot 2^{n-i}$, we have
\begin{align*}
\Big| \bigcup_{S\in \mathcal{S}'} S \Big| 
   = \sum_{i} \sum_{S\in \mathcal{S}_i'} |S|
   &\le \sum_{i} 2^{14} i^5 \cdot 2^{18} \cdot 2^{n-i}
   = 2^{32} \cdot 2^{n} \sum_{i} \frac{i^5}{2^{i}}
   < 2^{44} \cdot 2^{n},
\end{align*}
which is a contradiction (we used the fact that $\sum_{i} \frac{i^5}{2^i} < 2^{12}$). 
Therefore, there exists an index $i$ for which $|\mathcal{S}'_i| > 2^{14} i^5$.
We have $i > 0$ since otherwise the set of codimension zero would have been added to the tiling
in the first step.

Let $C = (a_1, \ldots, a_i, *, \ldots, *)$.
For each $\ell=1,2$, let $\mathcal{A}_{\ell}$ be the family
of cubes of level $\ell$ in our partial double tiling which are adjacent to $C$ and have codimension at most $i$.
By Proposition \ref{prop:K4_cuberelation}, we have
$|\mathcal{A}_{\ell}| \le i$ for each $\ell$.
For a cube $A \in \mathcal{A}_{\ell}$, we say that a set $S \in
\mathcal{S}_i'$ is \emph{bad} for $A$ if there are at least
$\frac{1}{(8i)^{5-\ell}}|S_{A}||S|$ blue edges between $S_{A}$ and $S$.
Otherwise, we say that $S$ is \emph{good} for $A$.
We claim that there are at most $2^{13} i^4$ sets in $\mathcal{S}'_i$
which are bad for each fixed $A$. 

Let $X_i = \bigcup_{S \in
\mathcal{S}'_i} S$ and note, by Proposition
\ref{prop:K4_preprocess}, that there are at most $|S_{A}|\cdot
2^{19} \cdot 2^{n-i}$ blue edges between the sets $S_{A}$ and $X_i$. Each set
$S \in \mathcal{S}'_i$ which is bad for $A$ accounts for at
least 
\[\frac{1}{(8i)^{5-\ell}}|S_{A}||S| \ge \frac{2^{18} \cdot 2^{n-i}}{(8i)^4}|S_{A}| = \frac{2^{6} \cdot 2^{n-i}}{i^4}|S_{A}|\] 
such blue edges (note that $|S| = 2^{18} \cdot 2^{n-i}$). Therefore, in
total, there are at most $2^{13} \cdot i^4$ sets in $\mathcal{S}'_i$ which are
in bad relation with $A$, as claimed above. Since
$|\mathcal{A}_{1} \cup \mathcal{A}_{2}| \le 2i$ and
$|\mathcal{S}'_i| > 2^{14} i^5$, there exists $S \in \mathcal{S}'_i$ which is
good for all cubes in $\mathcal{A}_{1} \cup \mathcal{A}_{2}$.

In order to show that $C$ satisfies (a) and (b), it suffices to
verify that $i \ge d$, since $i \ge d$ implies the fact that $C$ is
disjoint from all the other cubes in $\mathcal{C}_{1}$ (if $C$
intersects some other cube, then that cube must contain $C$ and
therefore also contains $(a_1, \ldots, a_n)$ by Proposition
\ref{prop:K3_cuberelation}). Furthermore, if this is the case,
assigning $C$ to $S$ is a proper assignment (thus we have
(c)). 

Now suppose, for the sake of contradiction, that $i < d$
and consider the time $t$ immediately after we last embedded a cube
of codimension at most $i$. At time $t$, since $(a_1, \ldots, a_n)$ was not
covered, the cubes in $\mathcal{C}$ are disjoint from $C$. Moreover,
the set of cubes of codimension at most $i$ which are adjacent to $C$ is the
same as at current time. Therefore, $C$ could have been added to the
tiling at time $t$ as well, and this contradicts the fact that we
always choose a cube of minimum codimension. Thus we have $i \ge d$,
as claimed.

\medskip

\noindent \textbf{Case 2}:  $(a_1, \ldots, a_n)$ is covered by a cube $C_1 \in \mathcal{C}_{1}$ but
not by a cube in $\mathcal{C}_{2}$.

In this case, in order to assign some cube $C \subset C_1$ containing
$(a_1, \ldots, a_n)$  to a subset of $S_{C_1}$ in $\mathcal{S}_{2}$, there are two types
of adjacency that one needs to consider. Let $C_2= (a_1, \ldots, a_{d}, *, \ldots, *)$ and temporarily consider it as a level $2$ cube. We first consider the cubes which are 1-adjacent to $C_2$ and remove
all the subsets of $S_{C_1}$ which are `bad' for these cubes. The reason we consider
cubes which are adjacent to $C_2$ instead of $C$ is because we do not
know what $C$ will be at this point. Note that, depending on
the choice of $C$, some of the cubes that are 1-adjacent to $C_2$
may not be 1-adjacent to $C$.
However, since we choose $C \subset C_2$ (as we will show later in the proof), 
the set of cubes which are 1-adjacent to $C$ will be a
subset of the set of cubes which are 1-adjacent to $C_2$. 
Moreover, if $C'$ is a cube 1-adjacent to both $C$ and $C_2$, 
then $\delta(C',C)=\delta(C',C_2)$.
Having removed all `bad' subsets, we have enough information to determine $C$.
Then, by considering the relation of the remaining subsets of $S_{C_1}$
to cubes which are 2-adjacent to $C$, we can find a set
$S_C$ that can be assigned to $C$. Note that 
unlike in the previous case,
even if $C'$ is a cube 2-adjacent to both $C$ and $C_2$,
we do not necessarily have $\delta(C',C)=\delta(C',C_2)$.
For this reason, it turns out to be crucial that 
we have already determined $C$ before the second step.

Let $d_1 = d(C_1)$ and let $\mathcal{A}^{(1)}$ be the subfamily of our partial embedding $\mathcal{C}$
consisting of
cubes which are $1$-adjacent to $C_2$. Note that, by Proposition \ref{prop:K4_cuberelation}(ii), there are at most $d_1$ cubes which are $1$-adjacent to $C_2$ in each level and thus $|\mathcal{A}^{(1)}| \le 2d_1$.
Let $\mathcal{F} = \{S \in \mathcal{S}_{2} : S \subset S_{C_1} \}$.
By Proposition \ref{prop:K4_preprocess},
we have $\Big| \bigcup_{S \in \mathcal{F}} S \Big| \ge \frac{|S_{C_1}|}{2}$.
We say that a set $S \in \mathcal{F}$ is \emph{bad}
for a cube $A \in \mathcal{A}^{(1)}$ if there are at least
$(8\delta_A)^{\ell(S)+\ell(S_A)-6} |S||S_A| = \frac{|S||S_A|}{(8\delta_A)^{4-\ell(A)}}$
blue edges between $S$ and $S_A$
(where $\delta_A = \max\{d_1, d_1(A)\})$. Otherwise, we say that $S$ is
\emph{good} for $A$.
For each fixed $A$, let $\mathcal{F}_A$ be the subfamily of $\mathcal{F}$
consisting of sets which are bad for $A$.
By the properness of the assignment up to this point, we know that there are
at most $\frac{|S_{C_1}||S_A|}{(8\delta_A)^{5-\ell(A)}}$ blue edges between
$S_{C_1}$ and $S_A$ for every $A \in \mathcal{A}^{(1)}$. Therefore, by
counting the number of blue edges between $S_{C_1}$ and $S_A$ in two ways,
we see that
\[ \sum_{S \in \mathcal{F}_A} \frac{|S||S_A|}{(8\delta_A)^{4-\ell(A)}}
    \le \frac{|S_{C_1}||S_A|}{(8\delta_A)^{5-\ell(A)}}, \]
from which we have $\sum_{S \in \mathcal{F}_A} |S| \le \frac{1}{8\delta_A}|S_{C_1}| \le \frac{1}{8d_1}|S_{C_1}|$.

Let $\mathcal{F}''$ be the subfamily of $\mathcal{F}$ of sets which are already
assigned to some cube in $\mathcal{C}_{2}$. There are no sets of
relative codimension zero in $\mathcal{F}''$ since this would imply
that $(a_1, \ldots, a_n)$ is covered by a cube in
$\mathcal{C}_{2}$. It thus follows that
for every $C \in \mathcal{C}_{2}$ such that $C \subset C_1$, we have
$|S_C| = 8 |C|$ and
\[ \Big| \bigcup_{S \in \mathcal{F}''} S \Big| = \sum_{S \in \mathcal{F}''} |S| \leq 8 \sum_{C\subset C_1}|C| \leq 
8|C_1| = 8 \cdot 2^{n-d_1}.  \]
Let $\mathcal{S}' = \mathcal{F} \setminus (\mathcal{F}'' \cup \bigcup_{A \in \mathcal{A}^{(1)}}\mathcal{F}_A)$
be the subfamily of $\mathcal{F}=\{S \in \mathcal{S}_2 : S \subset S_{C_1}\}$
of sets which are not assigned to any cubes yet and are good for
all the cubes in $\mathcal{A}^{(1)}$. Since $|S_{C_1}| \ge 2^{18}|C_1| = 2^{18} \cdot 2^{n-d_1}$,
we have
\begin{align*}
   \Big| \bigcup_{S \in \mathcal{S}'} S \Big|
   &\ge \Big| \bigcup_{S \in \mathcal{F}} S \Big| - \Big| \bigcup_{S \in \mathcal{F}''} S \Big| - \sum_{A \in \mathcal{A}^{(1)}} \Big| \bigcup_{S \in \mathcal{F}_A} S \Big| \\
   &\ge \frac{|S_{C_1}|}{2} - 8 \cdot 2^{n-d_1} - 2d_1 \cdot \frac{|S_{C_1}|}{8d_1} > 2^{15} \cdot 2^{n - d_1}.
\end{align*}
For each $i \ge 0$, let $\mathcal{S}'_i = \{S \in \mathcal{S}' : d_2(S) = i\} = \{S \in \mathcal{S}' : d(S) = d_1 + i \}$. Suppose that $|\mathcal{S}'_i| \le 128 i^3$ for all $i$. Then, since
we have $|S| = 8\cdot 2^{n-d_1 -i}$ for all $S \in \mathcal{S}'_i$,
\begin{align*}
\Big| \bigcup_{S\in \mathcal{S}'} S \Big| 
   = \sum_{i} \sum_{S\in \mathcal{S}_i'} |S|
   &\le \sum_{i} 128 i^3 \cdot 8 \cdot 2^{n-d_1-i}
   < 2^{15} \cdot 2^{n-d_1},
\end{align*}
which is a contradiction. Therefore, there exists an index $i$ for which $|\mathcal{S}'_i| > 128 \cdot i^3$.

Let $C = (a_1, \ldots, a_{d_1+i}, *, \ldots, *)$ and consider it as a level 2 cube.
By Proposition \ref{prop:K4_cuberelation}(ii),
there are at most $i$ cubes in $\mathcal{C}$
which have level 2 adjacency with $C$ and have codimension at most $d_1+i$. Let
$\mathcal{A}^{(2)}$ be the family consisting of these cubes.
For a cube $A \in \mathcal{A}^{(2)}$, we say that a set $S \in
\mathcal{S}_i'$ is \emph{bad} for $A$ if there are at least
$(8i)^{\ell(S_A) + \ell(S) - 6} |S_A||S| = \frac{1}{(8i)^2}|S_{A}||S|$ blue edges between $S_{A}$ and $S$.
Otherwise, we say that $S$ is \emph{good} for $A$.
We claim that there are at most $128i^2$ sets in $\mathcal{S}'_i$
which are bad for each fixed $A$. 

Let $X_i = \bigcup_{S \in
\mathcal{S}'_i} S$ and note, by Proposition
\ref{prop:K4_preprocess}, that there are at most $|S_{A}|\cdot
16 \cdot 2^{n-i}$ blue edges between the sets $S_{A}$ and $X_i$. Each set
$S \in \mathcal{S}'_i$ which is bad for $A$ accounts for at
least 
\[\frac{1}{(8i)^2}|S_{A}||S| = \frac{8 \cdot 2^{n-i}}{64i^2}|S_{A}| = \frac{2^{n-i}}{8i^2}|S_{A}|\]
such blue edges (note that $|S| = 8 \cdot 2^{n-i}$). Therefore, in
total, there are at most $128i^2$ sets in $\mathcal{S}'_i$ which are
bad for $A$, as claimed above. Since there are at most $i$ sets
in $\mathcal{A}^{(2)}$ and $|\mathcal{S}'_i| > 128i^3$, there exists a set $S \in \mathcal{S}'_i$ which is good for
all the cubes $A \in \mathcal{A}^{(2)}$.

In order to show that $C$ satisfies (a) and (b), it suffices to
verify that $d_1 + i \ge d$, since this implies the fact that $C$ is
disjoint from all the other cubes in $\mathcal{C}_{2}$ (note that
$C \subset C_1$ and if $C$
intersects some other cube of level 2, then that cube must contain $C$ and
therefore also contains $(a_1, \ldots, a_n)$ by Proposition
\ref{prop:K3_cuberelation}). Furthermore, if this is the case,
assigning $C$ to $S$ is a proper assignment (thus we have
(c)). 

Now suppose, for the sake of contradiction, that $d_1 + i < d$
and consider the time $t$ immediately after we last embedded a cube
of codimension at most $d_1 + i$. At time $t$, since $d_1 + i \ge d_1 = d(C_1)$, the cube
$C_1$ was already embedded and, since there are no
cubes of level 2 covering $(a_1, \ldots, a_n)$,
the cubes in $\mathcal{C}_{2}$ are disjoint from $C$. A cube in the partial embedding at time $t$, which is of codimension at most $d_1+i$, is $1$-adjacent to $C$ if and only if it is $1$-adjacent to $C_2$. Hence, the family of cubes which are adjacent to $C$ at time $t$ is a subfamily of $\mathcal{A}^{(1)} \cup \mathcal{A}^{(2)}$.
Therefore, $C$ could have been added to the tiling at time $t$ as well, contradicting the fact that we
always choose a cube of minimum codimension. Thus we have $d_1 + i \ge d$,
as claimed.
\end{proof}

Note that as an outcome of our algorithm, we obtain a tiling
$\mathcal{C}$ such that for every pair of adjacent cubes
$C, C' \in \mathcal{C}$, we have control on the number of blue edges between $S_C$ and
$S_{C'}$ (as given in the definition of proper assignment).

\subsection{Imposing a maximum degree condition} \label{sec:K4_maxdegree}

As in the triangle case, we now
impose certain maximum degree conditions between the sets $S_C$ for
$C \in \mathcal{C}$. It suffices to impose maximum degree conditions
between sets of level 2.

For a set $C \in \mathcal{C}_{2}$ of codimension $d = d(C)$,
and relative codimensions $d_1 = d_1(C), d_2 = d_2(C)$, recall
that we have a set $S_C \in \mathcal{S}$ such that $|S_C| \ge 8 \cdot 2^{n-d}$.
Let $\mathcal{A}_\rho$ be the family of cubes in $\mathcal{C}_{2}$ with
codimension at most $d$
which have level $\rho$ adjacency with $C$. For each $A \in \mathcal{A}_\rho$,
let $\delta_{A,\rho} = \max\{d_\rho, d_\rho(A)\}$.
By Proposition
\ref{prop:K4_cuberelation}(ii), we have $|\mathcal{A}_\rho| \le d_\rho$
for each $\rho=1,2$. For each $A \in \mathcal{A}_\rho$, there are
at most $\frac{1}{64\delta_{A,\rho}^2}|S_C||S_{A}|$ blue edges between $S_C$ and
$S_{A}$.

Now for $\rho=1,2$ and each $A \in \mathcal{A}_\rho$, remove
all the vertices in $S_C$ which have at
least $\frac{1}{8\delta_{A,\rho}}|S_{A}|$ blue neighbors in $S_{A}$ and let
$T_C$ be the subset of $S_C$ left after these removals. Since there
are at most $\frac{1}{64\delta_{A,\rho}^2}|S_C||S_{A}|$ blue edges between $S_C$
and $S_{A}$, we remove at most $\frac{|S_C|}{8\delta_{A,\rho}}$ vertices from $S_C$
for each set $A \in \mathcal{A}_\rho$. Thus the resulting set
$T_C$ is of size at least
\[ |T_C| \ge |S_C| - d_1 \cdot \frac{|S_C|}{8\delta_{A,1}} - d_2 \cdot \frac{|S_C|}{8\delta_{A,2}} \ge \frac{|S_C|}{2} \ge 4 \cdot 2^{n-d}. \]

For each $A \in \mathcal{A}_\rho$, all the vertices
in $T_C$ have blue degree at most $\frac{1}{8\delta_{A,\rho}}|S_{A}| \le
\frac{1}{4\delta_{A,\rho}}|T_{A}|$ in the set $T_{A}$. Thus we obtain the
following property.

\begin{quote}
\textbf{Maximum degree condition.} Let $C, C' \in \mathcal{C}_{2}$ be
a pair of cubes having level $\rho$ adjacency with $d(C) \ge d(C')$.
Then every
vertex in $T_C$ has at most $\frac{1}{4\delta_{\rho}}|T_{C'}|$ blue
neighbors in the set $T_{C'}$ (where $\delta_{\rho} = \max\{d_\rho(C), d_\rho(C')\}$).
\end{quote}

\subsection{Embedding the cube}

We now show how to embed $Q_n$. Recall that we found a double
tiling $\mathcal{C}$ of $Q_n$.
We will greedily embed the cubes in the level 2
tiling $\mathcal{C}_{2}$ one by one into their assigned sets from the family $\mathcal{S}_2$, in
decreasing order of their codimensions. If there are several cubes
of the same codimension, then we arbitrary choose the order between
them. 

Suppose that we are about to embed the cube $C \in
\mathcal{C}_{2}$. Let $d = d(C)$, $d_1 = d_1(C)$ and
$d_2 = d_2(C)$. We will greedily embed the vertices of $C$ into the set $T_C \subseteq S_C$. Suppose that we are about to embed $x \in C$ and let $f :
Q_n \rightarrow [N]$ denote the partial embedding of the cube $Q_n$
obtained so far. For $\rho = 1,2$, let $A_\rho$ be the set of
neighbors of $x$ which are already embedded and belong to a cube
other than $C$ that has level $\rho$ adjacency with $C$. Note that we have $|A_\rho| \le d_\rho$ for $\rho=1$ and $2$.
Since we have so far only embedded cubes of codimension at
least $d$, we have, for each $\rho=1,2$, that the vertices $f(v)$ for $v \in A_\rho$ have blue degree at most
$\frac{1}{4d_\rho}|T_C|$ in the set $T_C$, by the maximum degree
condition imposed in Section \ref{sec:K4_maxdegree}. Together, these
neighbors forbid at most
\[ d_1 \cdot \frac{1}{4d_1}|T_C| + d_2 \cdot \frac{1}{4d_2}|T_C| \le \frac{1}{2}|T_C|. \]
vertices of $T_C$ from being the image of $x$.

In addition, $x$ has at most $n-d$ neighbors which are already embedded and
belong to $C$. By Proposition \ref{prop:K4_preprocess}, for each such
vertex $v$, $f(v)$ has blue degree at most $\frac{2^{n-d}}{n}$ in the set
$T_C$. Together, these neighbors forbid at most
$2^{n-d}$ vertices of $T_C$ from being the image of $x$.
Finally, there are at most $2^{n-d}-1$ vertices in $T_C$ which are
images of some other vertex of $C$ that is already embedded. Therefore, the
number of vertices in $T_C$ into which we cannot embed $x$ is at most
\[ \frac{1}{2}|T_C| + 2^{n-d} + (2^{n-d}-1) < |T_C|, \]
where the inequality follows since $|T_C| \ge 4 \cdot 2^{n-d}$. Hence,
there exists a vertex in $T_C$ which we can choose as an image of
$x$ to extend the current partial embedding of the cube. Repeating
this procedure until we embed the whole cube $Q_n$
completes the proof.

\section{General case} \label{sec:general}

In this section, we further extend the arguments presented so far to prove the main theorem. 
The framework is very similar to that of the previous sections,
where we find multiple levels of tiling of $Q_n$. 
We begin by preprocessing the coloring to find families of sets $\mathcal{S} = \mathcal{S}_0 \cup
\mathcal{S}_1 \cup \ldots \cup \mathcal{S}_{s-2}$ such that, for $\ell \ge 1$, sets in $\mathcal{S}_{\ell}$
are subsets of sets in $\mathcal{S}_{\ell-1}$ which
do not contain blue $K_{s-\ell}$ (except for some special cases).
We refer to the sets in $\mathcal{S}_\ell$ as \emph{level $\ell$} sets and, for a 
set $S \in \mathcal{S}_\ell$, we let its \emph{level} be $\ell(S) = \ell$.
Note that, unlike in the cases where $s=3,4$, we have an additional level, level zero.
Level zero will consist of a single set $[N]$ and it is there merely for technical reasons.

We then seek a corresponding (multiple level) tiling $\mathcal{C}$ of the cube $Q_n$. 
Define an \emph{$(s-1)$-fold tiling} (or \emph{$(s-1)$-tiling}, for short) 
of $Q_n$ as a collection of $s-1$ tilings 
$\mathcal{C}_0 \cup \mathcal{C}_1 \cup \ldots \cup \mathcal{C}_{s-2}$,
where $\mathcal{C}_0$ is the trivial tiling consisting of the unique cube $Q_n$, and,
for all $\ell \ge 1$, $\mathcal{C}_{\ell}$ is a refined tiling of $\mathcal{C}_{\ell-1}$ (i.e.
for every $C \in \mathcal{C}_{\ell}$, there exists $C' \in \mathcal{C}_{\ell-1}$ such that
$C \subset C'$). We refer to cubes in $\mathcal{C}_\ell$ as \emph{level $\ell$ cubes} and,
for a cube $C \in \mathcal{C}_\ell$, define its \emph{level} as $\ell(C) = \ell$.
We will construct the tiling by finding cubes $C \in \mathcal{C}$ and assigning each of them
to some set $S_C \in \mathcal{S}$. Informally, this means that the subcube $C$ of $Q_n$
will be found in the $S_C$ part of our graph. Note that the trivial level zero cube $Q_n$ 
gets assigned to the trivial level zero set $[N]$ and this fits the heuristic.

For the rest of this section, we assume that $s \geq 5$.
Let $c = s^{15s}$ and suppose that $N \ge c^s 2^n = s^{15s^2} 2^{n}$.
We will later use the following estimate.

\begin{lemma} \label{lem:boundsum}
For every positive integer $s$, $\sum_{i=1}^{\infty} \frac{i^s}{2^i} \le 2s^{s}$.
\end{lemma}
\begin{proof}
Let $(x)_t=x(x-1)\ldots (x-t+1)$, $X_t = \sum_{i=1}^{\infty} \frac{i^t}{2^i}$ and $Y_t = \sum_{i=1}^{\infty} \frac{(i)_t}{2^i}$
for non-negative integers $t$. The {\it Stirling number $S(t,k)$ of the second kind} is the number of ways to partition a set of $t$ objects into $k$ non-empty subsets. These numbers satisfy the following well-known identity $x^t=\sum_{k=0}^t S(t,k)(x)_k$ (see, e.g., \cite{St02}, Chapter 1.4). This implies the identity $$X_t=\sum_{k=0}^t S(t,k)Y_k.$$ 

By taking the derivative $k$ times of both sides of the equality $(1-z)^{-1}=\sum_{i\geq 0} z^i$, note that
$$k!\cdot(1-z)^{-(k+1)}=\sum_{i\geq 1} i(i-1)\ldots(i-k+1)z^{i-k}.$$ By multiplying both sides by $z^k$ and
substituting $z=1/2$ we have that $Y_k=2k!$. This, together with the above identity, implies that $X_t= \sum_{k=0}^t 2k! S(t,k)$. Let $T_s$ be the number of partitions of a set of $s$ objects into labelled nonempty subsets. By counting over the size $k$ of the partition, we have the identity $T_s=\sum_{k=0}^s k!S(s,k)=X_s/2$. As each partition counted by $T_s$ is determined by the vector of labels of the sets containing each object, and there are at most $s$ such labels for each partition, we have $T_s \leq s^s$ and the desired inequality follows.  

(Although it will be not be needed, we remark that there is an explicit formula 
\[ X_t=2\sum_{k=0}^t\sum_{j=0}^k (-1)^{k-j}{k \choose j} j^t \] 
which follows from substituting in the well-known identity 
$S(t,k)=\frac{1}{k!}\sum_{j=0}^k (-1)^{k-j}{k \choose j} j^t$.)
\end{proof}

\subsection{Preprocessing the coloring} \label{sec:Ks_preprocess}

Let $\mathcal{S}_{0} = \{[N]\}$ and $[N]$ be the unique set of level zero
and codimension zero (denoted as $\ell(S) = 0$ and $d(S) = 0$).
We construct the levels one at a time. Once we finish constructing
$\mathcal{S}_{\ell-1}$, for each set $S' \in \mathcal{S}_{\ell-1}$, we use the
following procedure to construct sets belonging to the $\ell$-th level $\mathcal{S}_{\ell}$:

\begin{quote}
1. If $S'$ is not exceptional:
for each $d=0, 1, 2, \ldots, \log n + s\log c$,
if there exists a set $S \subset S'$ which induces a blue $K_{s-\ell}$-free graph of order
exactly $c^{s-\ell} \cdot 2^{n-d(S')-d}$, then arbitrarily choose one, add it to
the family $\mathcal{S}_{\ell}$ and remove the vertices of $S$ from
$S'$. We define the \emph{$\ell$-codimension} $d_\ell(S)$ of such a set as $d_\ell(S)=d$
and, for $i=1,\ldots,\ell-1$, we define the \emph{$i$-codimension} $d_i(S)$ as
$d_i(S) = d_i(S')$. Let the \emph{codimension} of $S$ be $d(S) = \sum_{i=1}^{\ell} d_i(S)$.
When there are no more such sets, continue to the next
value of $d$. If, after running through all values of $d$,
$\sum_{S \in \mathcal{S}_\ell, S \subset S'} |S| < \frac{|S'|}{2}$, then add the set of
remaining vertices to $\mathcal{S}_\ell$ and declare it to be an \emph{exceptional set}
with $\ell$-codimension zero (note that this set has size at least 
$\frac{|S'|}{2}=\frac{c^{s-(\ell-1)}\cdot 2^{n-d(S')}}{2} \ge c^{s-\ell}\cdot 2^{n-d(S')}$).

2. If $S'$ is exceptional: add $S'$ to the family $\mathcal{S}_{\ell}$ and define
its $\ell$-codimension as zero.
\end{quote}

Let $\mathcal{S} = \bigcup_{\ell=0}^{s-2} \mathcal{S}_{\ell}$ (we suppose that $\mathcal{S}$ is a multi-family
and if a set appears multiple times we distinguish them by their levels).
The following proposition is similar to Proposition \ref{prop:K4_preprocess}. We omit its proof.

\begin{proposition} \label{prop:Ks_preprocess}
\begin{enumerate}[(i)]
  \setlength{\itemsep}{1pt} \setlength{\parskip}{0pt}
  \setlength{\parsep}{0pt}
\item For $1 \le \ell \le s-2$ and a set $S' \in \mathcal{S}_{\ell-1}$, $\sum_{S \in \mathcal{S}_{\ell}, S \subset S'} |S| \ge \frac{|S'|}{2}$.
\item For integers $1 \le \ell \le s-2$ and $i \ge 1$, let $S'$ be a set of level $\ell-1$ and let
\[ X = \bigcup_{S \in \mathcal{S}_{\ell}, S \subset S', d_\ell(S) \ge i}S. \]
Then each vertex $v \in S'$ has at most $2c^{s-\ell} \cdot 2^{n-d(S')-i}$ blue neighbors in $X$.
\item For every set $S \in \mathcal{S}_{s-2}$, the subgraph induced by $S$ has maximum blue degree at most $\frac{2^{n-d(S)}}{n}$.
\end{enumerate}
\end{proposition}

\subsection{Tiling the cube} \label{sec:Ks_tiling}

In this subsection, we find an $(s-1)$-tiling of $Q_n$.
Recall that in the previous section, we had to control the blue edge densities
between adjacent cubes in the tiling. The parameter that governed the
control of these densities was defined in terms of the $\rho$-codimension,
where $\rho$ was the level of adjacency of these cubes. Below we
generalize this concept.

\begin{definition}
Let $\mathcal{C}$ be an $(s-1)$-tiling and let $C, C'\in \mathcal{C}$ be
two adjacent cubes.
\begin{enumerate}[(i)]
  \setlength{\itemsep}{1pt} \setlength{\parskip}{0pt}
  \setlength{\parsep}{0pt}
\item The \emph{level of adjacency} $\rho(C,C')$ is the minimum $\ell$ such that
the cubes of level $\ell$ containing $C$ and $C'$ are distinct. We say that $C$ and
$C'$ are $\rho$-adjacent if $\rho(C,C') = \rho$.
\item The \emph{dominating parameter} $\delta(C,C')$ is $\max\{d_\rho(C), d_\rho(C')\}$,
where $\rho = \rho(C,C')$.
\end{enumerate}
\end{definition}

Note that the level of adjacency of two cubes $C$ and $C'$ is at most $\min\{\ell(C), \ell(C')\}$.
The following proposition is similar to Proposition \ref{prop:K4_cuberelation} and we omit its proof.

\begin{proposition} \label{prop:Ks_cuberelation}
Let $\mathcal{C}$ be an $(s-1)$-tiling and let $C \in \mathcal{C}$ be a
level $\ell$ special cube of codimension $d$.
For each $\rho=1,2,\ldots,\ell$ and $\ell'=1,2,\ldots,s-2$, the number
of special cubes of level $\ell'$ and codimension at most $d$ which
are $\rho$-adjacent to $C$ is at most $d_\rho(C)$.
\end{proposition}

Our $(s-1)$-tiling $\mathcal{C}$ of the cube $Q_n$ will be constructed in
correspondence with the family $\mathcal{S}$ constructed in Section
\ref{sec:Ks_preprocess}. We will construct $\mathcal{C}$
by finding cubes of the tiling one by one. We slightly abuse
notation and use $\mathcal{C}$ also to denote the `partial' $(s-1)$-tiling,
where only part of the cube $Q_n$ is covered.
At each step, we find
a subcube $C$ which covers some non-covered part of $Q_n$ and
assign it to some set $S_C \in \mathcal{S}$. We say that such an
assignment is \emph{proper} if the following properties hold.

\begin{quote}
{\bf Proper assignment.}

1. $\ell(C) = \ell(S_C)$ and, for $\ell = \ell(C)$, we have $d_\ell(C) = d_\ell(S_C)$.

2. If $C \subset C'$ for $C' \in \mathcal{C}$, then $S_{C} \subset S_{C'}$.

3. Suppose that $C$ is adjacent to some cube $C'$ already in the
tiling, where $C'$ is assigned to $S_{C'}$, and that
$C$ and $C'$ have level $\rho$ adjacency. Then the number
of blue edges in the bipartite graph induced by $S_C$ and $S_{C'}$
is at most $|S_C||S_{C'}| \cdot (4s^2\delta(C,C'))^{\ell(C) + \ell(C') - 2s}$.
\end{quote}

We use $S_C$ to denote the set in $\mathcal{S}$ to which $C$ is
assigned. Our algorithm for finding the tiling $\mathcal{C}$ and the
corresponding sets in $\mathcal{S}$ is as follows.

\begin{quote}
\textbf{Tiling Algorithm.} At each step, consider all possible cubes
$C$ which 
\begin{itemize}
\item[(a)] can be added to $\mathcal{C}$ to extend the partial tiling, 

\item[(b)]
have $d(C) \ge d(C')$ for all $C' \in \mathcal{C}$, and 

\item[(c)]
for which there exists a set $S \in \mathcal{S}$ which has not yet been
assigned and such that assigning $C$ to $S$ gives a proper assignment. 
\end{itemize}
Take a cube $C_0$ of minimum codimension satisfying (a), (b) and (c) and add it
to the tiling. Assign $C_0$ to the set $S_{C_0} \in \mathcal{S}$ given by (c).
\end{quote}

The following proposition shows that the algorithm will terminate successfully.

\begin{proposition} \label{prop:Ks_tiling}
If the tiling is not complete, then the algorithm always chooses a cube from
a non-empty collection.
\end{proposition}
\begin{proof}
In the beginning, the algorithm will take $Q_n$ as the level 0 tiling
and will assign it to $[N]$, which is the unique set of level 0.
Now suppose that in the previous step we embedded some cube of codimension $d$.
Since the tiling is not complete, there exists a
vertex $(a_1, \ldots, a_n) \in Q_n$ which is not covered $s-1$ times.
Suppose that this vertex is covered $\ell$ times for $\ell \leq s-2$ and let
$C_0, \ldots, C_{\ell-1}$ be the cubes of each level that cover it
(note that $C_0 \supseteq C_1 \supseteq \ldots \supseteq C_{\ell-1}$). Let $C_{\ell}= (a_1, \ldots, a_{d}, *, \ldots, *)$. We temporarily consider $C_{\ell}$ as a level $\ell$ cube.

In this case, in order to assign some cube $C \subset C_{\ell-1}$ containing
$(a_1, \ldots, a_n)$ to a subset of $S_{C_{\ell-1}}$ in $\mathcal{S}_{\ell}$,
we first consider the cubes which are $\rho$-adjacent to $C_{\ell}$ for
$\rho \le \ell-1$ and remove all the subsets of $S_{C_{\ell-1}}$ which
are `bad' for these cubes. The reason we consider
cubes which are adjacent to $C_{\ell}$ instead of $C$ is because we do not
know what $C$ will be at this point. Note that, depending on
the choice of $C$, some of the cubes that are $\rho$-adjacent to $C_{\ell}$
may not be $\rho$-adjacent to $C$.
However, since we choose $C \subset C_{\ell}$ (as we will show later in the proof), the set of cubes which
are $\rho$-adjacent to $C$ will be a
subset of the set of cubes which are $\rho$-adjacent to $C_{\ell}$.
Moreover, for $\rho \le \ell-1$, if $C'$ is a cube $\rho$-adjacent to both $C$ and $C_{\ell}$, 
then $\delta(C',C)=\delta(C',C_{\ell})$.
Having removed these `bad' subsets, we have enough information to determine $C$.
Then, by considering the relation of the remaining subsets of $S_{C_{\ell-1}}$
to cubes which are $\ell$-adjacent to $C$, we can find a set
$S_C$ that can be assigned to $C$.
Note that unlike in the previous case, for $\rho \ge \ell$,
even if $C'$ is a cube $\rho$-adjacent to both $C$ and $C_{\ell}$,
we do not necessarily have $\delta(C',C)=\delta(C',C_{\ell})$. 
For this reason, it turns out to be crucial that we
have already determined $C$ before the second step.

For each $\rho \leq \ell - 1$ and $i$, let $\mathcal{A}^{(\rho)}_i$ be the family of
level $i$ cubes from our partial embedding $\mathcal{C}$
consisting of cubes which are $\rho$-adjacent to $C_{\ell}$. Let
$\mathcal{A}^{(\rho)} = \bigcup_{i} \mathcal{A}^{(\rho)}_i$. By Proposition \ref{prop:Ks_cuberelation},
we have $|\mathcal{A}^{(\rho)}_i| \le d_\rho(C_{\ell-1})$ for each $i$ and thus
$|\mathcal{A}^{(\rho)}| \le s \cdot d_\rho(C_{\ell-1})$.
Let $\mathcal{F} = \{S \in \mathcal{S}_{\ell} : S \subset S_{C_{\ell-1}} \}$.
By Proposition \ref{prop:Ks_preprocess},
we have $\Big| \bigcup_{S \in \mathcal{F}} S \Big| \ge \frac{|S_{C_{\ell-1}}|}{2}$.
We say that a set $S \in \mathcal{F}$ is \emph{bad}
for a cube $A \in \mathcal{A}^{(\rho)}_i$ if there are at least
$(4s^2\delta_A)^{\ell(S)+\ell(S_A)-2s} |S||S_A| = (4s^2\delta_A)^{\ell+i-2s} |S||S_A|$
blue edges between $S$ and $S_A$
(where $\delta_A = \delta(C_{\ell-1}, A)$ is the dominating parameter
of $C_{\ell-1}$ and $A$). Otherwise, we say that $S$ is \emph{good} for $A$.
For each fixed $A$, let $\mathcal{F}_A$ be the subfamily of $\mathcal{F}$
consisting of sets which are bad for $A$.
By the properness of the assignment up to this point, we know that there are
at most $(4s^2\delta_A)^{\ell+i-1-2s} |S_{C_{\ell-1}}||S_A|$ blue edges between
$S_{C_{\ell-1}}$ and $S_A$ for every $A \in \mathcal{A}^{(\rho)}_i$. Therefore, by
counting the number of blue edges between $S_{C_{\ell-1}}$ and $S_A$ in two ways,
we see that
\[ \sum_{S \in \mathcal{F}_A} (4s^2\delta_A)^{\ell+i-2s} |S||S_A|
    \le (4s^2\delta_A)^{\ell+i-1-2s} |S_{C_{\ell-1}}||S_A|, \]
from which we have $\sum_{S \in \mathcal{F}_A} |S| \le \frac{1}{4s^2\delta_A}|S_{C_{\ell-1}}|
\le \frac{1}{4s^2d_\rho(C_{\ell-1})}|S_{C_{\ell-1}}|$.

Let $\mathcal{F}''$ be the subfamily of $\mathcal{F}$ of sets which are already
assigned to some cube in $\mathcal{C}_{\ell}$. There can be no sets of
relative codimension zero in $\mathcal{F}''$ since this would imply
that $(a_1, \ldots, a_n)$ is covered by a cube in
$\mathcal{C}_{\ell}$. It thus follows that
for every $C \in \mathcal{C}_{\ell}$ such that $C \subset C_{\ell-1}$, we have
$|S_C| = c^{s-\ell} |C|$ and
\[ \Big| \bigcup_{S \in \mathcal{F}''} S \Big| = \sum_{S \in \mathcal{F}''} |S| \le c^{s-\ell} \cdot |C_{\ell-1}|.  \]
Let $\mathcal{S}' = \mathcal{F} \setminus (\mathcal{F}'' \cup \bigcup_\rho \bigcup_{A \in \mathcal{A}^{(\rho)}}\mathcal{F}_A)$
be the subfamily of $\mathcal{F}$ of sets which are not assigned to any cubes yet and are good for
all the cubes in $\mathcal{A}^{(\rho)}$ for all $\rho \leq \ell - 1$.
Since $|S_{C_{\ell-1}}| \ge c^{s-\ell+1}|C_{\ell-1}|$, we have
\begin{align*}
   \Big| \bigcup_{S \in \mathcal{S}'} S \Big|
   &\ge \Big| \bigcup_{S \in \mathcal{F}} S \Big| - \Big| \bigcup_{S \in \mathcal{F}''} S \Big| - \sum_{\rho} 
\sum_{A \in \mathcal{A}^{(\rho)}} \Big| \bigcup_{S \in \mathcal{F}_A} S \Big| \\
   &\ge \frac{|S_{C_{\ell-1}}|}{2} - c^{s-\ell} \cdot |C_{\ell-1}| - \sum_{\rho} s d_\rho(C_{\ell-1}) \cdot \frac{|S_{C_{\ell-1}}|}{4s^2d_\rho(C_{\ell-1})} > 
\frac{c^{s-\ell+1}}{5} \cdot |C_{\ell-1}|.
\end{align*}
For each $i \ge 0$, let $\mathcal{S}'_i = \{S \in \mathcal{S}' : d_\ell(S) = i\} = \{S \in \mathcal{S}' : d(S) = d(C_{\ell-1}) + i \}$. Suppose that $|\mathcal{S}'_i| \le (4s^2i)^{2s+1}$ for all $i$. Then, since
we have $|S| = c^{s-\ell} 2^{n-d(C_{\ell-1})-i} = c^{s-\ell} \cdot 2^{-i} |C_{\ell-1}|$ for all $S \in \mathcal{S}'_i$,
\begin{align*}
\left|\bigcup_{S \in  \mathcal{S}'} S \right|
   = \sum_{i}\sum_{S \in \mathcal{S}_i'} \left|S \right|
   &\le \sum_{i} (4s^2i)^{2s+1} \cdot c^{s-\ell} \cdot 2^{-i} |C_{\ell-1}|.
\end{align*}
By Lemma \ref{lem:boundsum}, we have 
$$\sum_{i}(4s^2i)^{2s+1} 2^{-i} =(2s)^{4s+2}\sum_{i}i^{2s+1}/ 2^i \leq (2s)^{4s+2} \cdot 2(2s+1)^{2s+1}<s^{15s}/5=c/5\,.$$
Therefore,  $\left|\bigcup_{S \in  \mathcal{S}'} S \right| < \frac{c^{s-\ell+1}}{5} \cdot |C_{\ell-1}|$,
which is a contradiction. Thus there exists an index $i$ for which $|\mathcal{S}'_i| > (4s^2i)^{2s+1}$.

Let $C = (a_1, \ldots, a_{d(C_{\ell-1})+i}, *, \ldots, *)$ and consider it as a level $\ell$ cube.
By Proposition \ref{prop:Ks_cuberelation},
there are at most $s \cdot i$ cubes in $\mathcal{C}$ with codimension at most $d(C)$
which have level $\ell$ adjacency with $C$ (remember that these may have level higher than $\ell$). Let
$\mathcal{A}^{(\ell)}$ be the family consisting of these cubes and
note that the dominating parameter of $C$ and $A$ for $A \in \mathcal{A}^{(\ell)}$
is always $i$.
For a cube $A \in \mathcal{A}^{(\ell)}$, we say that a set $S \in
\mathcal{S}_i'$ is \emph{bad} for $A$ if there are at least
$(4s^2i)^{\ell(S_A) + \ell(S) - 2s} |S_A||S| = (4s^2i)^{\ell(S_A) + \ell - 2s} |S_A||S|$
blue edges between $S_{A}$ and $S$.
Otherwise, we say that $S$ is \emph{good} for $A$. We claim that there are at most $(4s^2i)^{2s}$ sets in $\mathcal{S}'_i$ which are
bad for each fixed $A$.

Let $X_i = \bigcup_{S \in \mathcal{S}'_i} S$ and note, by Proposition
\ref{prop:Ks_preprocess}, that there are at most $|S_{A}|\cdot
2c^{s-\ell} \cdot 2^{n-d(C)}$ blue edges between the sets $S_{A}$ and $X_i$. Each set
$S \in \mathcal{S}'_i$ which is bad for $A$ accounts for at
least $(4s^2i)^{\ell(S_A) + \ell - 2s} |S_A||S| \ge
\frac{c^{s-\ell} \cdot 2^{n-d(C)}}{(4s^2i)^{2s-1}}|S_{A}|$ such blue edges (note that $|S| = c^{s-\ell} \cdot 2^{n-d(C)}$ and $\ell\geq 1$). Therefore, in
total, there are at most $(4s^2i)^{2s}$ sets in $\mathcal{S}'_i$ which are
bad for $A$, as claimed above. Since there are at most $si$ cubes
in $\mathcal{A}^{(\ell)}$ and $|\mathcal{S}'_i| > (4s^2i)^{2s+1}$, there exists a set $S \in \mathcal{S}'_i$ which is good for
all the cubes $A \in \mathcal{A}^{(\ell)}$.

In order to show that $C$ satisfies (a) and (b), it suffices to
verify that $d(C) \ge d$, since this implies the fact that $C$ is
disjoint from all the other cubes of level at least $\ell$
(note that $C \subset C_{\ell-1}$ and if $C$
intersects some other cube of level at least $\ell$, then that cube must contain $C$ and
therefore also contains $(a_1, \ldots, a_n)$ by Proposition
\ref{prop:K3_cuberelation}). Furthermore, if this is the case,
assigning $C$ to $S$ is a proper assignment (thus we have
(c)). 

Now suppose, for the sake of contradiction, that $d(C) < d$
and consider the time $t$ immediately after we last embedded a cube
of codimension at most $d(C)$. At time $t$, since $d(C) \ge d(C_{\ell-1})$, the cube
$C_{\ell-1}$ was already embedded and, since there are no
cubes of level $\ell$ covering $(a_1, \ldots, a_n)$,
the cubes in $\mathcal{C}_{\ell}$ are disjoint from $C$. For $\rho \leq \ell-1$, 
a cube in the partial tiling at time $t$, which is of codimension at most $d(C)$, 
is $\rho$-adjacent to $C$ if and only if it is 
$\rho$-adjacent to $C_{\ell}$. Hence, the family of cubes which are adjacent to $C$ at time $t$ is
a subfamily of $\bigcup_{\rho=1}^{\ell} \mathcal{A}^{(\rho)}$.
Therefore, $C$ could have been added to the
tiling at time $t$ as well, contradicting the fact that we
always choose a cube of minimum codimension. Thus we have $d(C) \ge d$,
as claimed.
\end{proof}

Note that as an outcome of our algorithm, we obtain a tiling
$\mathcal{C}$ such that for every pair of adjacent cubes
$C, C' \in \mathcal{C}$, we have control on the number of blue edges between $S_C$ and
$S_{C'}$ (as given in the definition of proper assignment).

\subsection{Imposing a maximum degree condition} \label{sec:Ks_maxdegree}

As in the previous sections, we now
impose certain maximum degree conditions between the sets $S_C$ for
$C \in \mathcal{C}_{s-2}$.
For a set $C \in \mathcal{C}_{s-2}$ of codimension $d = d(C)$
and relative codimensions $d_\ell = d_\ell(C), 1 \leq \ell \leq s-2$, recall
that we have a set $S_C \in \mathcal{S}$ such that $|S_C| \ge c^2 \cdot 2^{n-d}$.
Let $\mathcal{A}^{(\rho)}$ be the family of cubes in $\mathcal{C}_{s-2}$ with
codimension at most $d$ which have level $\rho$ adjacency with $C$.
By Proposition
\ref{prop:Ks_cuberelation}, we have $|\mathcal{A}^{(\rho)}| \le d_\rho(C)$
for each $\rho=1,\ldots,s-2$. For each $A \in \mathcal{A}^{(\rho)}$, let
$\delta_A = \delta(C,A)$ and note that $\delta_A=\max\{d_\rho(A),d_\rho(C)\}\geq d_\rho(C)$.
Also, since $\ell(A)+\ell(C)-2s = -4$, there are
at most $\frac{1}{(4s^2\delta_A)^4}|S_C||S_{A}|$ blue edges between $S_C$ and
$S_{A}$.

Now for $\rho=1,\ldots,s-2$, and each $A \in \mathcal{A}^{(\rho)}$, remove
all the vertices in $S_C$ which have at
least $\frac{1}{4s^2\delta_A}|S_{A}|$ blue neighbors in $S_{A}$ and let
$T_C$ be the subset of $S_C$ left after these removals. Since there
are at most $\frac{1}{(4s^2\delta_A)^4}|S_C||S_{A}|$ blue edges between $S_C$
and $S_{A}$, we remove at most $\frac{|S_C|}{(4s^2\delta_A)^3}$ vertices from $S_C$
for each set $A \in \mathcal{A}^{(\rho)}$. Thus the resulting set
$T_C$ is of size at least
\[ |T_C| \ge |S_C| - \sum_{\rho=1}^{s-2} d_\rho \cdot \frac{|S_C|}{(4s^2\delta_A)^3}
         \ge \frac{|S_C|}{2} \ge c \cdot 2^{n-d}. \]

For each $A \in \mathcal{A}^{(\rho)}$, all the vertices
in $T_C$ have blue degree at most $\frac{1}{4s^2\delta_{A}}|S_{A}| \le
\frac{1}{2s^2\delta_{A}}|T_{A}|$ in the set $T_{A}$. Thus we obtain the
following property.

\begin{quote}
\textbf{Maximum degree condition.} Let $C, C' \in \mathcal{C}_{s-2}$ be
a pair of cubes having level $\rho$ adjacency with $d(C) \ge d(C')$.
Then every
vertex in $T_C$ has at most $\frac{1}{2s^2 \delta(C,C')}|T_{C'}|$ blue
neighbors in the set $T_{C'}$.
\end{quote}

\subsection{Embedding the cube}

We now show how to embed $Q_n$. Recall that we found an $(s-1)$-tiling $\mathcal{C}$ of $Q_n$.
We will greedily embed the cubes in the level $s-2$
tiling $\mathcal{C}_{s-2}$ one by one into their assigned sets from the family $\mathcal{S}_{s-2}$, in
decreasing order of their codimensions. If there are several cubes
of the same codimension, then we arbitrary choose the order between
them. 

Suppose that we are about to embed the cube $C \in
\mathcal{C}_{s-2}$. Let $d = d(C)$ and $d_\ell = d_\ell(C)$, for
$\ell=1,\ldots,s-2$. We will greedily embed the vertices of $C$ into $T_C \subseteq S_C$. Suppose that we are about to embed $x \in C$ and let $f :
Q_n \rightarrow [N]$ denote the partial embedding of the cube $Q_n$
obtained so far. For each $\rho$, let $A_{\rho}$ be the set of
neighbors of $x$ which are already embedded and belong to a cube
other than $C$ that has level $\rho$ adjacency with $C$. 
Note that we have $|A_{\rho}| \le d_\rho$ for every $\rho$.
Since we have so far only embedded cubes of codimension at
least $d$, for each $\rho$, the vertices $f(v)$ for $v \in A_{\rho}$ have blue degree at most
$\frac{1}{2s^2d_\rho}|T_C|$ in the set $T_C$, by the maximum degree
condition imposed in Section \ref{sec:Ks_maxdegree}. Together, these
neighbors forbid at most
\[ \sum_{\rho=1}^{s-2} d_\rho \cdot \frac{1}{2s^2d_\rho}|T_C| \le \frac{1}{2}|T_C|. \]
vertices of $T_C$ from being the image of $x$.

In addition, $x$ has at most $n-d$ neighbors which are already embedded and
belong to $C$. By Proposition \ref{prop:Ks_preprocess}, for each such
vertex $v$, $f(v)$ has blue degree at most $\frac{2^{n-d}}{n}$ in the set
$T_C$. Together, these neighbors forbid at most
$2^{n-d}$ vertices of $T_C$ from being the image of $x$.
Finally, there are at most $2^{n-d}-1$ vertices in $T_C$ which are
images of some other vertex of $C$ that is already embedded. Therefore, the
number of vertices in $T_C$ into which we cannot embed $x$ is at most
\[ \frac{1}{2}|T_C| + 2^{n-d} + (2^{n-d}-1) < |T_C|, \]
where the inequality follows since $|T_C| \ge c \cdot 2^{n-d}$.
Hence, there exists a vertex in $T_C$ which we can choose as an image of
$x$ to extend the current partial embedding of the cube. Repeating
this procedure until we embed the whole cube $Q_n$
completes the proof.

\section{Small separators, forbidden minors and Ramsey goodness} \label{sec:separator}

Let $\mathcal{G}_H$ be the family of graphs $G$ which do not contain an $H$-minor. In order to show that this family is $s$-good, we wish to apply the following result of Nikiforov and Rousseau \cite{NR09}.

\begin{theorem} \label{NikRou}
For every $s \geq 3$, $d \geq 1$ and $0 < \gamma < 1$, there exists $\eta > 0$ such that the class $\mathcal{G}$ of $d$-degenerate graphs $G$ with a $(|V(G)|^{1-\gamma}, \eta)$-separator is $s$-good.
\end{theorem}

There are two conditions here that need to be verified in order to gain the conclusion of Theorem \ref{forbidminor}. Firstly, we need to show that the graphs in $\mathcal{G}_H$ have bounded degeneracy. This was first proved by Mader \cite{M68}. Later,   Kostochka \cite{K82, K84} and Thomason \cite{T84} independently established the following bound, which is tight apart from the constant factor. More recently, the asymptotic value of $c$ was determined by Thomason \cite{T01}.

\begin{theorem} \label{KT}
There exists a constant $c > 0$ such that any graph with average degree at least $c h \sqrt{\log h}$ contains a $K_h$-minor. 
\end{theorem}

Secondly, we need to show that the graphs in $\mathcal{G}_H$ have appropriate separators. We will use the following result of Alon, Seymour and Thomas \cite{AST90}. Note that, in the particular case of planar graphs, such a separator theorem was proved much earlier by Lipton and Tarjan \cite{LT79}.

\begin{theorem} \label{AST}
Let $G$ be a graph on $n$ vertices containing no $K_h$-minor. Then $G$ has an $(h^{3/2} n^{1/2}, \frac{2}{3})$-separator. 
\end{theorem}

We are now ready to prove our main result about $\mathcal{G}_H$, which we recall from the introduction.

\begin{theorem} \label{forbid}
For every fixed graph $H$, the class $\mathcal{G}_H$ of graphs $G$ which do not contain an $H$-minor is $s$-good for all $s \geq 3$.
\end{theorem}

\begin{proof}
In order to apply Theorem \ref{NikRou}, it is enough to verify that, for any $\eta > 0$, any sufficiently large graph $G$ in $\mathcal{G}_H$ has bounded degeneracy and a $(|V(G)|^{2/3}, \eta)$-separator.

Suppose that the graph $H$ has $h$ vertices. Note, by Theorem \ref{KT}, that any graph with average degree at least $c h \sqrt{\log h}$ contains a $K_h$-minor and, hence, an $H$-minor. This implies that the average degree of every subgraph of $G$ is at most $c h \sqrt{\log h}$. In turn, this easily implies that the degeneracy of $G$ is at most $c h \sqrt{\log h}$.

Suppose that $G$ has $n$ vertices. To show that a sufficiently large graph $G$ from the class $\mathcal{G}_H$ contains an $(n^{2/3}, \eta)$-separator $T$, we begin by applying Theorem \ref{AST} to conclude that there is an $(h^{3/2} n^{1/2}, \frac{2}{3})$-separator. Note, by taking unions of small components if necessary, that this gives a decomposition of the vertex set of $G$ into three sets $T, A$ and $B$ such that $|T| \leq h^{3/2}n^{1/2}$, $|A|, |B| \leq \frac{2}{3} n$ and there are no edges between $A$ and $B$.

We will prove, by induction on $i$, that for $n$ sufficiently large depending on $i$, there is a separator $T_i$ of size at most $2^i h^{3/2} n^{1/2}$ that splits the vertex set of $G$ into $2^i$ sets $U_{i, 1}, \dots, U_{i, 2^i}$, each of size at most $\left(\frac{2}{3}\right)^i n$, so that there are no edges between any distinct sets $U_{i,a}$ and $U_{i,b}$. By the previous paragraph, the result holds for $i = 1$. Now suppose that it holds for $i$. We will show that a similar conclusion follows for $i+1$.

To begin, we apply Theorem \ref{AST} within each of the sets $U_{i,j}$ to conclude that there is a decomposition of $U_{i,j}$ into sets $T_{i,j}, A_{i,j}$ and $B_{i,j}$ such that $|T_{i,j}| \leq h^{3/2} |U_{i,j}|^{1/2}$, 
\[|A_{i,j}|, |B_{i,j}| \leq \frac{2}{3} |U_{i,j}| \leq \left(\frac{2}{3}\right)^{i+1} n\] 
and there are no edges between $A_{i,j}$ and $B_{i,j}$. 
We let the collection $\{U_{i+1,j}\}_{j=1}^{2^{i+1}}$ consist of all sets of the form
$A_{i,j}$ and $B_{i,j}$ for $j=1,\ldots, 2^i$.
This collection has size $2^{i+1}$, each of the sets has size at most $\left(\frac{2}{3}\right)^{i+1} n$ and there are no edges between distinct sets $U_{i+1,a}$ and $U_{i+1, b}$. If we also let $T_{i+1}$ be the union of $T_i$ and $T_{i,j}$, for $1 \leq j \leq 2^i$, we have
\[|T_{i+1}| \leq |T_i| + \sum_{j=1}^{2^i} |T_{i,j}| \leq 2^i h^{3/2} n^{1/2} + \sum_{j=1}^{2^i} h^{3/2} |U_{i,j}|^{1/2} \leq  2^i h^{3/2} n^{1/2} + 2^i h^{3/2} n^{1/2} =  2^{i+1} h^{3/2} n^{1/2}.\]
Therefore, the induction holds.

If we now apply this result with $i = 2 \log \eta^{-1}$ and $n \geq h^9\eta^{-12}$, we see, since $\left(\frac{2}{3}\right)^i \leq \eta$ and $n^{2/3} \geq 2^i h^{3/2} n^{1/2}$, that $G$ has an $(n^{2/3}, \eta)$-separator, as required. The result follows. 
\end{proof}

Let $\mathcal{K}$ be the collection of graphs $K$ for which there is a proper vertex coloring in $\chi(K)$ colors such that at least two of the color classes have size one. The full result of Nikiforov and Rousseau (namely, Theorem 2.2 of \cite{NR09}) says that for any $K \in \mathcal{K}$, $d \geq 1$ and $0 < \gamma < 1$ there exists $\eta > 0$ such that the class $\mathcal{G}$ of $d$-degenerate graphs $G$ with a $(|V(G)|^{1-\gamma}, \eta)$-separator is $K$-good. This may in turn be used to show that for any $H$ the family of graphs $\mathcal{G}_H$ is $K$-good for all $K \in \mathcal{K}$. However, $\mathcal{G}_H$ is not $K$-good for all graphs $K$. This follows from the observation mentioned in the introduction that $K_{1,t}$ is not $K_{2,2}$-good for any $t$.

%We conjecture that the same should be true for any graph $K$ and not just those in the special class $\mathcal{K}$.

The family of graphs $\mathcal{G}_H$ with forbidden $H$-minor is not the only class of graphs known to have small separators. For example, several geometric separator theorems are known (see, for example, \cite{FP10, MTTV97}) saying that the class of intersection graphs formed by certain collections of bodies have small separators. In any of these cases, Theorem \ref{NikRou} will also allow us to show that the classes are $K$-good for any $K \in \mathcal{K}$, though in some cases we may have to further restrict the class in order to obtain the required degeneracy condition.

\section{Concluding remarks}

The question of determining whether the cube is $s$-good for any $s \geq 3$ is only one of two well-known questions of Burr and Erd\H{o}s regarding Ramsey numbers and the cube. The other \cite{BE75, CG98} is the question of determining whether the Ramsey number $r(Q_n) := r(Q_n, Q_n)$ of the 
cube with itself is linear in the number of vertices in $Q_n$.

Beginning with Beck \cite{B83}, who proved that $r(Q_n) \leq 2^{c n^2}$, a large number of papers \cite{GRR00, GRR01, KR01, S01, S07} have considered this question, with the current best bound being $r(Q_n) \leq n 2^{2n + 5}$, due to Fox and Sudakov \cite{FS07}. That is, $r(Q_n) \leq |Q_n|^{2 + o(1)}$. This (almost) quadratic bound for $r(Q_n)$ was obtained using a careful application of dependent random choice \cite{FS11}. It seems likely that in order to improve the bound one will have to improve this latter technique. However, we do not rule out the possibility that some of the embedding lemmas used in the current paper could also be of use.

Suppose $H$ is an $N$-vertex graph with chromatic number $r$, maximum degree $O(1)$ and bandwidth $o(N)$. Bollob\'as and Koml\'os conjectured that every $N$-vertex graph $G$ which does not contain $H$ as a subgraph has minimum degree at most $(1-\frac{1}{r}+o(1))N$. This conjecture was recently verified by B\"ottcher, Schacht and Taraz \cite{BST09}. It is natural to wonder to what extent the bounded maximum degree condition on $H$ can be relaxed. In particular, is it true that every $Q_n$-free graph on $2^n$ vertices has minimum degree at most $(\frac{1}{2}+o(1))2^n$? Note that $Q_n$ is a bipartite graph on $N=2^n$ vertices, which is $n$-regular with $n=\log N$, and has bandwidth $O\left(N/\sqrt{\log N}\right)$. It appears that new techniques would have to be developed to handle this problem. The proof of the Bollob\'as-Koml\'os conjecture uses the regularity lemma and the blow-up lemma, giving quantitative estimates that are too weak to embed spanning subgraphs which are as dense as cubes. A positive solution would likely lead to a proof of the Burr-Erd\H{o}s conjecture discussed above that cubes have linear Ramsey number. 

%It would be of great interest to know whether the methods of this paper could be improved to give an approximate result of the form $r(Q_n, K_s) \leq (s - 1 + \epsilon) 2^n$, even in the case of triangles. Such a result would likely be a necessary first step in resolving the original question of Burr and Erd\H{o}s.

%We conjecture that for each odd cycle $C_{2k+1}$, the family of cubes is $C_{2k+1}$-good. In the case $k=1$, this is simply the special case of the Burr-Erd\H{o}s conjecture that the family of cubes is $3$-good. We think the problem might be more tractable for longer odd cycles and could be a step towards proving that the family of cubes is $3$-good. 

\medskip

\noindent \textbf{Acknowledgements}. We would like to thank the
two anonymous referees for their valuable comments.

\medskip 

\noindent \textbf{Note added in proof.} Recently, the main conjecture studied in this paper, that $r(Q_n, K_s) = (s-1)(2^n-1) + 1$ for $s$ fixed and $n$ sufficiently large, was resolved by Fiz Pontiveros, Griffiths, Morris, Saxton and Skokan \cite{FGMSS13, FGMSS132}. Their proof builds upon the techniques developed in this paper.


\begin{thebibliography}{}

\bibitem{ABS12}
{P. Allen, G. Brightwell and J. Skokan,} {Ramsey-goodness -- and otherwise,} 
{\it Combinatorica} {\bf 33} (2013), 125--160.

\bibitem{AST90}
{N. Alon, P. Seymour and R. Thomas,} {A separator theorem for nonplanar graphs,} {\it J. Amer. Math. Soc.} {\bf 3} (1990), 801--808.

\bibitem{B83}
{J. Beck,} {An upper bound for diagonal Ramsey numbers,} {\it Studia Sci. Math. Hungar.}
{\bf 18} (1983), 401--406.

\bibitem{BPTW10}
{J. B\"ottcher, K. Pruessman, A. Taraz and A. W\"urfl,} {Bandwidth, expansion, treewidth, separators, and universality for bounded degree graphs,} {\it European J. Combin.} {\bf 31} (2010), 1217--1227.

\bibitem{BST09}
{J. B\"ottcher, M. Schacht and A. Taraz,} {Proof of the bandwidth conjecture of Bollob\'as and
Koml\'os,} {\it Math. Ann.} {\bf 343} (2009), 175--205.

\bibitem{B96}
{S. Brandt,} {Expanding graphs and Ramsey numbers,} 
{available at Freie Universit\"at, Berlin preprint server, ftp://ftp.math.fu-berlin.de/pub/math/publ/pre/1996/pr-a-96-24.ps} (1996).

\bibitem{B66}
{W. G. Brown,} {On graphs that do not contain a Thomsen graph,} {\it Canad. Math. Bull.} {\bf 9} (1966), 281--285.

\bibitem{B81}
{S. A. Burr,} {Ramsey numbers involving graphs with long suspended paths,} {\it J. London
Math. Soc. (2)} {\bf 24} (1981), 405--413.

\bibitem{B87}
{S. A. Burr,} {What can we hope to accomplish in generalized Ramsey theory?,} {\it Discrete Math.} {\bf 67} (1987), 215--225.

\bibitem{BE75}
{S. A. Burr and P. Erd\H{o}s,} {On the magnitude of generalized Ramsey numbers for graphs,} {in: Infinite and Finite Sets I (Keszthely, 1973), Colloq. Math. Soc. Janos Bolyai, Vol. 10,} 214--240, North-Holland, Amsterdam, 1975.

\bibitem{BE83}
{S. A. Burr and P. Erd\H{o}s,} {Generalizations of a Ramsey-theoretic result of Chv\'atal,} {\it J.
Graph Theory} {\bf 7} (1983), 39--51.

\bibitem{BEFRS85}
{S. A. Burr, P. Erd\H{o}s, R. J. Faudree, C. C. Rousseau and R. H. Schelp,} {The Ramsey number for the pair complete bipartite graph--graph of limited degree,} {in: Graph theory with applications to algorithms and computer science (Kalamazoo, Mich., 1984),} 163--174, Wiley, New York, 1985.

\bibitem{BEFRS89}
{S. A. Burr, P. Erd\H{o}s, R. J. Faudree, C. C. Rousseau and R. H. Schelp,} {Some complete bipartite graph--tree Ramsey numbers,} {\it Ann. Discrete Math.} {\bf 41} (1989), 79--90. 

\bibitem{CG98}
{F. Chung and R. L. Graham,} {\bf Erd\H{o}s on Graphs. His Legacy of
Unsolved Problems}, A K Peters, Ltd., Wellesley, MA, 1998.

\bibitem{C77}
{V. Chv\'atal,} {Tree-complete graph Ramsey numbers,} {\it J. Graph Theory} {\bf 1} (1977), 93.

\bibitem{CH72}
{V. Chv\'atal and F. Harary,} {Generalized Ramsey theory for graphs, III. Small off-diagonal
numbers,} {\it Pacific J. Math.} {\bf 41} (1972), 335--345.

\bibitem{FGMSS13}
{G. Fiz Pontiveros, S. Griffiths, R. Morris, D. Saxton and J. Skokan,} {On the Ramsey number of the triangle and the cube,} {\it preprint}.

\bibitem{FGMSS132}
{G. Fiz Pontiveros, S. Griffiths, R. Morris, D. Saxton and J. Skokan,} {The Ramsey number of the clique and the hypercube,} {\it preprint}.

\bibitem{FP10}
{J. Fox and J. Pach,} {A separator theorem for string graphs and its applications,} {\it Combin. Probab. Comput.} {\bf 19} (2010), 371--390.

\bibitem{FS07}
{J. Fox and B. Sudakov,} {Density theorems for bipartite graphs and
related Ramsey-type results,} {\it Combinatorica} {\bf 29} (2009), 153-196.

\bibitem{FS11}
{J. Fox and B. Sudakov,} {Dependent random choice,} {\it Random Structures Algorithms} {\bf 38} (2011), 68--99.

\bibitem{GRR00}
{R. L. Graham, V. R\"odl and A. Ruci\'nski,} {On graphs with linear
Ramsey numbers,} {\it J. Graph Theory} {\bf 35} (2000), 176--192.

\bibitem{GRR01}
{R. L. Graham, V. R\"odl and A. Ruci\'nski,} {On bipartite graphs with linear Ramsey numbers,} {\it Combinatorica} {\bf 21} (2001), 199--209.

\bibitem{K82}
{A. Kostochka,} {The minimum Hadwiger number for graphs with a given mean degree of vertices,} {\it Metody Diskret. Analiz.} {\bf 38} (1982), 37--58.

\bibitem{K84}
{A. Kostochka,} {Lower bound of the Hadwiger number of graphs by their average degree,} {\it Combinatorica} {\bf 4} (1984), 307--316.

\bibitem{KR01}
{A. Kostochka and V. R\"odl,} {On graphs with small Ramsey numbers,} {\it J. Graph Theory} {\bf 37} (2001), 198--204.

\bibitem{LT79}
{R. J. Lipton and R. E. Tarjan,} {A separator theorem for planar graphs,} {\it SIAM J. Appl. Math.} {\bf 36} (1979), 177--189.

\bibitem{M68} 
{W. Mader,} {Homomorphies\"atze f\"ur Graphen,} {\it Math. Ann.} {\bf 178} (1968), 154--168.

\bibitem{MTTV97}
{G. L. Miller, S.-H. Teng, W. Thurston and S. A. Vavasis,} {Separators for
sphere-packings and nearest neighbor graphs,} {\it J. ACM} {\bf 44} (1997), 1--29.

\bibitem{NR09}
{V. Nikiforov and C. C. Rousseau,} {Ramsey goodness and beyond,} {\it Combinatorica} {\bf 29} (2009), 227--262.

\bibitem{RS04}
{N. Robertson and P. D. Seymour,} {Graph minors. XX. Wagner's conjecture,} {\it J. Combin. Theory Ser. B} {\bf 92} (2004), 325--357.

\bibitem{S01}
{L. Shi,} {Cube Ramsey numbers are polynomial,} {\it Random Structures Algorithms} {\bf 19} (2001), 99--101.

\bibitem{S07}
{L. Shi,} {The tail is cut for Ramsey numbers of cubes,} {\it Discrete Math.} {\bf 307} (2007), 290--292. 

\bibitem{St02}
{R. Stanley,} {\bf Enumerative Combinatorics, Volume 1}, Cambridge Studies in Advanced Mathematics 49, Cambridge University Press, Cambridge, UK, 2002.

\bibitem{T84}
{A. Thomason,} {An extremal function for contractions of graphs,} {\it Math. Proc. Cam. Phil. Soc.} {\bf 95} (1984), 261--265.

\bibitem{T01} 
{A. Thomason,} {The extremal function for complete minors,} {\it J. Combin. Theory Ser. B} {\bf 81} (2001), 318--338. 
\end{thebibliography}
\end{document}